\documentclass[12pt]{amsart}
\usepackage{amsfonts, amsthm, amsmath, amssymb, graphicx, color}
\usepackage{tikz}

\newcommand{\Z}{\mathbb Z}
\newcommand{\N}{\mathbb N}
\newcommand{\Q}{\mathbb Q}
\newcommand{\R}{\mathbb R}

\title[A new class of $\alpha$-Farey maps and normal numbers]{A new class of $\alpha$-Farey maps and an application to normal numbers}
\author{K.~Dajani}
\address{Department of Mathematics, Utrecht University, P.O.~Box 80010, 3508 TA Utrecht, The Netherlands}
\email{k.dajani1@uu.nl}
\author{C.~Kraaikamp}
\address{Delft University of Technology, EWI (DIAM), Mekelweg 4, 2628 CD Delft, The Netherlands}
\email{c.kraaikamp@tudelft.nl}
\author{H.~Nakada}
\address{Department of Mathematics, Keio University, 3-14-1 Hiyoshi, Kohoku-ku, Yokohama 223-8522, Japan}
\email{nakada@math.keio.ac.jp}
\author{R.~Natsui}
\address{Department of Mathematics, Japan Women’s University, 2-8-1 Mejirodai, Bunkyou-ku, Tokyo, 112-8681, Japan}
\email{natsui@fc.jwu.ac.jp}

\newtheorem{thm}{Theorem}

\newtheorem{lem}{Lemma}
\newtheorem{prop}{Proposition}

\newtheorem{rem}{Remark}

\newcommand{\confrac}[2]{
\frac{\displaystyle{
\strut\hfill{#1}\hfill\;\vrule}}
{\displaystyle{
 \strut\vrule\;\hfill{#2}\hfill}}}

\begin{document}
\begin{abstract}
We define two types of the $\alpha$-Farey maps $F_{\alpha}$ and $F_{\alpha, \flat}$ for $0 < \alpha < \tfrac{1}{2}$,
which were previously defined only for $\tfrac{1}{2} \le \alpha \le 1$ by R.~Natsui (2004). Then, for each $0 < \alpha < \tfrac{1}{2}$, we construct the natural extension maps on the plane and show that the natural extension of $F_{\alpha, \flat}$ is metrically isomorphic to the natural extension of the original Farey map.  As an application,
we show that the set of normal numbers associted with $\alpha$-continued fractions
does not vary by the choice of $\alpha$, $0 < \alpha < 1$.  This extends the result
by C.~Kraaikamp and H.~Nakada (2000).
 \end{abstract}
\maketitle

\renewcommand{\thefootnote}{\fnsymbol{footnote}}
\footnote[0]{{\it MSC2020:} Primary 11K50, 37A10; Secondary 11J70, 37A44. }
\footnote[0]{{\it Keywords:} Farey map, $\alpha$-continued fraction expansions, natural extension, normals numbers}

\renewcommand{\thefootnote}{\arabic{footnote}}

\section{Introduction}\label{Introduction}
The main purpose of this paper is to extend the notion of the $\alpha$-Farey map to $0 < \alpha < \tfrac{1}{2}$, and discuss its properties with applications.  We start with a simple introduction of the theory of the regular continued fraction map.

Let $x$ be a real number, then it is well-known that the \emph{simple} or \emph{regular continued fraction} (RCF) \emph{expansion} of $x$ yields a finite (if $x\in\Q$) or infinite (if $x\in\R\setminus\Q$) sequence of rational convergents $(p_n/q_n)$ with extremely good approximation properties; see e.g.~\cite{DK2002, HW, IK, K, P, RS}. The RCF-expansion of $x$ can be obtained using the so-called \emph{Gauss map}\ $G:[0,1]\to [0,1)$, defined as follows:
\[
G(x) = \begin{cases}
\displaystyle{\frac{1}{x} - \left\lfloor \frac{1}{x} \right\rfloor}, & \text{if $x\neq 0$}; \\
0, & \text{if $x = 0$}.
\end{cases}
\]
For $0<x<1$, the digits (or: partial quotients) $a_n=a_n(x)$ of the RCF-expansion of $x$ are defined for $n\geq 1$ by $a_{n}(x) = \lfloor \frac{1}{G^{n-1}(x)}\rfloor$,
where $\lfloor \frac{1}{0} \rfloor = \infty$ and $\frac{1}{\infty}=0$. For $x\in\R$, we define $a_0=a_0(x)=\lfloor x\rfloor$; if $x\not\in (0,1)$, we define for $n\geq 1$ the digit $a_n(x)$ by setting $a_n(x):=a_n(x-a_0)$.

It is well-known that for $x\in (0,1)$ the simple continued fraction expansion of $x$ easily follows from the above definitions of $G$ and $a_n(x)$:
\[
x = \confrac{1}{a_{1}(x)} + \confrac{1}{a_{2}(x)} + \cdots + \confrac{1}{a_{n}(x)}
+ \cdots .
\]
We put
\[
\frac{p_{n}(x)}{q_{n}(x)} = \confrac{1}{a_{1}(x)} + \confrac{1}{a_{2}(x)} + \cdots +
\confrac{1}{a_{n}(x)} ,
\]
where $p_n(x)$, $q_n(x)\in\N$ and where we assume that $(p_{n}(x), q_{n}(x)) = 1$. This rational number $p_n(x)/q_n(x)$ is called the $n$-th principal convergent of $x$.  It is also well known that for $n \geq 1$,
\[
\begin{cases}
p_{n}(x) \,\,\, = & a_{n}(x) p_{n-1}(x) + p_{n-2}(x); \\
q_{n}(x) \,\,\, = & a_{n}(x) q_{n-1}(x) + q_{n-2}(x),
\end{cases}
\]
with $p_{-1}(x) = 1$, $p_{0}(x) = 0$, $q_{-1}(x) = 0$, $q_{0}(x) = 1$.  If $a_{n}(x)\geq 2$ for $n\geq 1$, then
\begin{equation}\label{eq:1}
\frac{\ell \cdot p_{n-1}(x) + p_{n-2}(x) }{\ell \cdot q_{n-1}(x) + q_{n-2}(x) },
\,\,\, 1 \le \ell < a_{n}(x)
\end{equation}
is called the $(n, \ell)$-mediant (or intermediate) convergent of $x$.

Setting
$$
\Theta_n(x) = q_n^2\left| x-\frac{p_n}{q_n}\right| ,\quad \text{for $n\geq 0$},
$$
one can easily show that for all irrational $x$ and all $n\geq 1$ one has that $0<\Theta_n(x)<1$; see e.g.~\cite{DK2002, IK}. Several classical results on these \emph{approximation coefficients} $\Theta_n(x)$ have been obtained for all $n\geq 1$ and all irrational $x$; just to mention a few:
$$
\min \{ \Theta_{n-1}(x),\Theta_n(x)\} <\tfrac{1}{2},\,\, \text{ (Vahlen, 1913)}
$$
and
$$
\min \{ \Theta_{n-1}(x),\Theta_n(x), \Theta_{n+1}(x)\} <\tfrac{1}{\sqrt{5}},\,\, \text{  (Borel, 1903)}.
$$
Borel's result is a consequence of
$$
\min \{ \Theta_{n-1}(x),\Theta_n(x),\Theta_{n+1}(x)\} <\tfrac{1}{\sqrt{a_{n+1}^2+4}},
$$
which was obtained independently by various authors; see also Chapter 4 in~\cite{DK2002}. That the sequence $(p_n(x)/q_n(x))$ converges extremely fast to $x$ follows from $0<\Theta_n(x)<1$ for all $n$, and the fact that the sequence $(q_n(x))$ grows exponentially fast. An old result by Legendre further underlines the Diophantine qualities of the RCF: let $x\in\R$, and let $p\in\Z$, $q\in\N$, such that $(p,q)=1$, and suppose we moreover have that
$$
\left| x-\frac{p}{q}\right| < \frac{1}{2} \frac{1}{q^2},
$$
then $p/q$ is a RCF-convergent of $x$. I.e., there exists an $n$ such that $p=p_n(x)$ and $q=q_n(x)$. Here the constant $1/2$ is best possible. In 1904, Fatou stated (and this was published in 1918 by Grace; see~\cite{G}), that if $\left| x-\frac{p}{q}\right| < \frac{1}{q^2}$, then $p/q$ is either an RCF-convergent, or an extreme mediant; i.e., an $(n, \ell)$-mediant convergent of $x$ from~(\ref{eq:1}) with $\ell =1$ or $\ell =a_n-1$. Further refinements of this result can be found in~\cite{BJ}.\smallskip\

The $(n, \ell)$-mediant convergents of $x$ from~(\ref{eq:1}) can be obtained by the so-called Farey-map $F$. The notion of the Farey map was introduced in 1989 by S.~Ito in~\cite{I} and by M.~Feigenbaum, I.~Procaccia and T.~Tel in~\cite{Fe} independently.  In particular, the metric properties of $F$ were discussed by Ito in~\cite{I}; see also~\cite{Bl, Bo, BY, DKS}.

To introduce $F$, we write $G$ as the composition of two maps: an \emph{inversion} $R: (0,1]\to [1,\infty)$ and a translation $S: [1,\infty)\to (0,1]$, defined as:
$$
R(x)=\frac{1}{x},\quad \text{for $x\in (0, 1]$},
$$
and
$$
S(x)=x-k,\quad \text{if $x\in [k, k+1)$, for some $k\in\N$}.
$$
If we furthermore define that $0\mapsto 0$, we clearly have that $G(x)= (S\circ R)(x)$ for $x\in (0,1]$. Note that the latter map $S$ can be written as the $k$-fold composition of a map $S_1: [1,\infty)\to [0,\infty)$, defined as $S_{1}(x)=x-1$ for $x\geq 1$, so that $S(x) = (\underbrace{S_{1} \circ \cdots \circ S_{1}}_k)(x)$.

Next, we extend the inversion $R$ to $[1,\infty)$: $R(x) = \frac{1}{x}$, for $x \in [1, \infty)$ so that we map $[1,\infty)$ bijectively on the bounded interval $(0, 1]$.  With this extended definition of the map $R$, we define the map $F:(0,1]\to [0,1)$ as a ``slow continued fraction map,'' given by:
$$
F(x) = \begin{cases}
(R\circ S_{1} \circ R)(x) &=\,\,\, {\displaystyle{\frac{x}{1 -x}}},\quad \text{if $x\in [0, 1/2)$}; \\
&  \\
G(x) = (S_1\circ R)(x) &=\,\,\, {\displaystyle{\frac{1-x}{x}}}, \quad \text{if $x \in [1/2, 1]$;}
\end{cases}
$$
see also~\cite{DKS}, where the relation between $F$ and the so-called \emph{Lehner continued fraction} is investigated.\smallskip\

For a given matrix $A = \begin{pmatrix} a_{11} & a_{12} \\ a_{21} & a_{22} \end{pmatrix} \in  \mbox{GL}(2, \Z)$, we define its associated linear fractional transformation as:
\begin{equation}\label{eq:2}
A(x) = \frac{a_{11}x + a_{12}}{a_{21}x + a_{22}}, \quad \text{for $x \in \R$}.
\end{equation}
The map $F$ ``yields'' the mediant convergents together with the principal (i.e., RCF) convergents in the following manner.  For each $x\in (0, 1)$, $F(x)$ is either $\tfrac{x}{1-x}$ or $\tfrac{1-x}{x}$, which is a linear fractional transformation associated with matrices
$\begin{pmatrix} 1 & 0 \\ -1 & 1 \end{pmatrix}$ and $\begin{pmatrix} -1 & 1 \\
1 & 0 \end{pmatrix}$, respectively.  We put
\[
A_{n}= A_{n}(x) = \begin{cases}
\begin{pmatrix} 1 & 0 \\ 1 & 1 \end{pmatrix}  \left(= \begin{pmatrix} 1 & 0 \\ -1 & 1 \end{pmatrix}^{-1}\right), & \text{if $0 < F^{n-1}(x) < \frac{1}{2}$}; \\
\begin{pmatrix} 0 & 1 \\ 1 & 1 \end{pmatrix} \left(= \begin{pmatrix} -1 & 1 \\ 1 & 0 \end{pmatrix}^{-1}\right), &  \text{if $\frac{1}{2} < F^{n-1}(x) < 1$},
\end{cases}
\]
thus yielding a sequence of matrices $(A_{n}:n \geq 1)$.   Viewing this sequence as a sequence of linear fractional transformations, we obtain a sequence of rationals $(t_{n}: n\geq 1)$ with $t_{n} = (A_{1}A_{2} \cdots A_{n})(-\infty)$, for each $n\geq 1$.  It is not hard to see that this sequence is
\begin{eqnarray*}
&&\frac{1}{1}, \frac{1}{2}, \ldots , \frac{1}{a_{1}-1}, \\
&&\frac{0}{1} = \frac{p_{0}}{q_{0}}, \frac{1}{a_{1} + 1} = \frac{1 \cdot p_{1} + p_{0}}{1\cdot q_{1} + q_{0}}, \frac{2 \cdot p_{1} + p_{0}}{2\cdot q_{1} + q_{0}}, \ldots ,
\frac{(a_{2}-1) \cdot p_{1} + p_{0}}{(a_{2}-1)\cdot q_{1} + q_{0}}, \\
&&\frac{p_{1}}{q_{1}}, \frac{1 \cdot p_2 + p_1}{1\cdot q_2 + q_1}, \frac{2 \cdot p_2 + p_1}{2\cdot q_2 + q_1}, \ldots, \frac{(a_3-1) \cdot p_2 + p_1}{(a_3-1)\cdot q_2 + q_1}, \\
&& {\phantom{XXXXXXXX}}\vdots \\
&&\frac{p_{n-1}}{q_{n-1}}, \frac{1 \cdot p_{n} + p_{n-1}}{1 \cdot q_{n} + q_{n-1}},  \ldots,
\frac{\ell \cdot p_{n} + p_{n-1}}{\ell \cdot q_{n} + q_{n-1}}, \ldots , \frac{(a_{n+1} -1) \cdot p_{n} + p_{n-1}}{(a_{n+1}-1) \cdot q_{n} + q_{n-1}}, \\
&&\frac{p_{n}}{q_{n}}, \frac{1 \cdot p_{n+1}+ p_{n}}{1 \cdot q_{n+1} + q_{n}}, \ldots\ldots\ldots ,
\end{eqnarray*}
i.e.,  we have the sequence of the mediant convergents together with the principal convergents of $x$.  We will find it again in \S2 as a special case of H.~Nakada's $\alpha$-expansions from~\cite{N1981}, with $\alpha = 1$.\smallskip\

Apart from the \emph{regular continued fraction expansion} there is a bewildering amount of other continued fraction expansion: continued fraction expansions with \emph{even} (or \emph{odd}) partial quotients, the \emph{optimal continued fraction expansion}, the Rosen fractions, and many more. In this paper we will look at a family of continued fraction algorithms, introduced by Nakada in 1981 in~\cite{N1981} with the natural extensions as planer maps. These continued fraction expansions are parameterized by a parameter $\alpha \in (0,1]$, the case $\alpha = 1$ being the RCF.  After their introduction, the natural extension of the Gauss map played an important role in solving a conjecture by Hendrik Lenstra, which was previously proposed by Wolfgang Doeblin (see~\cite{BJW}, and also~\cite{DK2002, IK} for more details on the proof and various corollaries of this \emph{Doeblin-Lenstra} conjecture).  The notion of the natural extension planer maps lead to various generalization, e.g.\ the so-called $S$-expansions, introduced by C.~Kraaikamp in~\cite{K1991}. The papers  mentioned here, and various other papers at the time dealt with the case $\tfrac{1}{2}\leq \alpha\leq 1$. At that time, there was no discussion on $\alpha$-continued fractions for $0<\alpha< \tfrac{1}{2}$ except for a 1999 paper by P.~Moussa, A.~Casa and S.~Marmi (\cite{M-C-M}), dealing with $\sqrt{2}-1<\alpha< \tfrac{1}{2}$.  Later on, after two papers published in 2008 by L.~Luzzi and Marmi (\cite{L-M}), and Nakada and R.~Natsui (\cite{N-N-2}), the interest to work on $\alpha$-continued fractions was rekindled, but then for parameters $\alpha\in (0,\tfrac{1}{2})$; see e.g.~\cite{K-S-S, dJK}.\smallskip\

In 2004, Natsui introduced and studied the so-called $\alpha$-Farey maps $F_{\alpha}$ in~\cite{Nats-1,Nats-2} for parameters $\alpha\in [\tfrac{1}{2},1)$. These maps $F_{\alpha}$ relate to the $\alpha$-expansion maps $G_{\alpha}$ from~\cite{N1981} as the Farey-map $F$ relates to the Gauss-map $G$. In this paper we investigate these $\alpha$-Farey maps $F_{\alpha}$ for $0 < \alpha < \tfrac{1}{2}$.\smallskip\

Recall from~\cite{N1981} that the $\alpha$-continued fraction map $G_{\alpha}$, for $0 < \alpha \leq 1$, is defined\footnote{We refer to papers by A.~Abrams, 
S.~Katok and I.~Ugarcovici \cite{A-K-U} and by S.~Katok and I.~Ugarcovici \cite{K-U-1, K-U-2, K-U-3}, where a similar idea is applied to  a 
 different type ($2$-parameter family) of continued fraction maps.} as follows. Let $\alpha\in (0,1]$ fixed, then for $x\in \mathbb I_{\alpha} = [\alpha - 1, \alpha)$ we define the map $G_{\alpha}$ as:
\begin{equation}\label{eq:3}
G_{\alpha}(x) = \begin{cases}
-\frac{1}{x} - \lfloor -\frac{1}{x} + 1 - \alpha\rfloor , & \text{if $x < 0$}; \\[3pt]
\,\, 0, & \text{if $x = 0$};\\
\frac{1}{x} - \lfloor \frac{1}{x} + 1 - \alpha\rfloor , & \text{if $x > 0$.}
\end{cases}
\end{equation}
For $ x \in \mathbb I_{\alpha}$, we put $a_{\alpha, n}(x) = \left\lfloor \frac{1}{|G_{\alpha}^{n-1}(x)|}  + 1 - \alpha\right\rfloor$ and $\varepsilon_{\alpha, n}(x)= \text{sgn} (x)$. Then we have for $x \in \mathbb I_{\alpha}\setminus \{ 0\}$ that:
$$
G_{\alpha}(x) = \frac{\varepsilon_{\alpha, n}(x)}{x} - a_{\alpha, n}(x).
$$
From this one easily finds that:
\[
x =   \confrac{\varepsilon_{\alpha, 1}(x)}{a_{\alpha, 1}(x)} +
      \confrac{\varepsilon_{\alpha, 2}(x)}{a_{\alpha, 2}(x)} + \cdots
      + \confrac{\varepsilon_{\alpha, n}(x)}{a_{\alpha, n}(x)} + \cdots ,
\]
which we call the $\alpha$-continued fraction expansion of $x$.
We define the $n$-th principal convergent as
\[
\frac{p_{\alpha, n}(x)}{q_{\alpha, n}(x)} =
  \confrac{\varepsilon_{\alpha, 1}(x)}{a_{\alpha, 1}(x)} +
      \confrac{\varepsilon_{\alpha, 2}(x)}{a_{\alpha, 2}(x)} + \cdots
      + \confrac{\varepsilon_{\alpha, n}(x)}{a_{\alpha, n}(x)} ,\qquad \text{for $n\geq 1$},
\]
where $p_{\alpha, n}(x)\in\Z$, $q_{\alpha, n}(x)\in\N$ and $(p_{\alpha, n}(x), q_{\alpha, n}(x) ) = 1$. Moreover, whenever $a_{\alpha, n}(x)\geq 2$ we also define the mediant convergents as
\[
\frac{\ell \cdot p_{\alpha, n-1}(x) + \varepsilon_{\alpha, n}(x) p_{\alpha, n-2}(x) } {\ell \cdot q_{\alpha, n-1}(x) + \varepsilon_{\alpha, n}(x)q_{\alpha, n-2}(x) }, \quad \text{for $1\leq\ell < a_{n}(x)$}.
\]
To get these mediant convergents, we consider the Farey type map $F_{\alpha}$, and as in the case $\alpha =1$ we show how it is related with $G_{\alpha}$; note that $G_1=G$ and $F_1=F$. As in the case $\alpha = 1$, we consider \emph{inversions} and a \emph{translation}. The inversions are now defined by
$$
R_{-}(x) = - \frac{1}{x},\,\, \text{for $x\in [\alpha -1, 0)$},\quad \text{and}\quad R(x) = \frac{1}{x}, \,\, \text{for $x >0$},
$$
while the translation is now defined by
$$
S_{1}(x)=x -1,\quad \text{for $x >\alpha$}.
$$
Of course, we again define that $0\mapsto 0$. From this process, we have the map $F_{\alpha}$ defined on $\left[\alpha -1, \frac{1}{\alpha}
\right]$ by
\begin{equation}\label{eq:4}
F_{\alpha}(x) = \begin{cases}
(R\circ S_{1}\circ R_{-})(x) & =\,\, -\frac{x}{1+x},\,\, \text{ if $x \in [\alpha -1, 0)$}; \\
(R\circ S_{1}\circ R)(x)     & =\,\,\,\, \frac{x}{1-x},\,\,\,\, \text{ if $x \in [0, \frac{1}{1 + \alpha}]$}; \\
(S_{1} \circ R)(x)           & =\,\,\,\, \frac{1-x}{x},\,\,\,\,  \text{ if $x \in (\frac{1}{1 + \alpha}, \frac{1}{\alpha}]$};
\end{cases}
\end{equation}
see Figure~\ref{figure1}.
\begin{figure}
\begin{center}
\begin{tikzpicture}[scale=1.4, domain=-0.75:4]
%  \draw[very thin,color=gray] (-0.1,-1.1) grid (3.9,3.9);
\draw (-0.75,0) -- (4,0);
\draw (0,-0.75) -- (0,4);
\draw[thick] plot[smooth] file {Figure1DataLeftFunction.tex};
\draw[thick] plot[smooth] file {Figure1DataMiddleFunction.tex};
\draw[thick] plot[smooth] file {Figure1DataRightFunction.tex};

\node at (-0.05, -0.1){\tiny $0$};

\node at (1, -0.1){\tiny $1$};
\draw (1,-0.05) -- (1,0.05);

\node at (0.1, 0.25){\tiny $\alpha$};
\draw (-0.05,0.25) -- (0.05,0.25);
\draw[dotted] (0.15,0.25) -- (0.8,0.25);

\node at (-0.75, -0.10){\tiny $\alpha - 1$};
\draw (-0.75,-0.05) -- (-0.75,0.05);
\draw[dotted] (-0.75,0) -- (-0.75,3);

\node at (-0.25,-0.75){\tiny $\alpha - 1$};
\draw (-0.05,-0.75) -- (0.05,-0.75);
\draw[dotted] (0,-0.75) -- (4,-0.75);

\node at (0.23, 3){\tiny $\frac{1-\alpha}{\alpha}$};
\draw (-0.05,3) -- (0.05,3);
\draw[dotted] (-0.75,3) -- (0,3);

\node at (4.05,-0.15){\tiny $\frac{1}{\alpha}$};
\draw (4,-0.05) -- (4,0.05);
\draw[dotted] (4,-0.75) -- (4,0);

\node at (0.12,4){\tiny $\frac{1}{\alpha}$};
\draw (-0.05,4) -- (0.05,4);
\draw[dotted] (0.22,4) -- (0.8,4);

\node at (0.8,-0.17){\tiny $\frac{1}{\alpha +1}$};
\draw (0.8,-0.05) -- (0.8,0.05);
\draw[dotted] (0.8,0) -- (0.8,4);
\end{tikzpicture}
\end{center}
\caption{The map $F_{\alpha}$ for $\alpha = \tfrac{1}{4}$}
\label{figure1}
\end{figure}
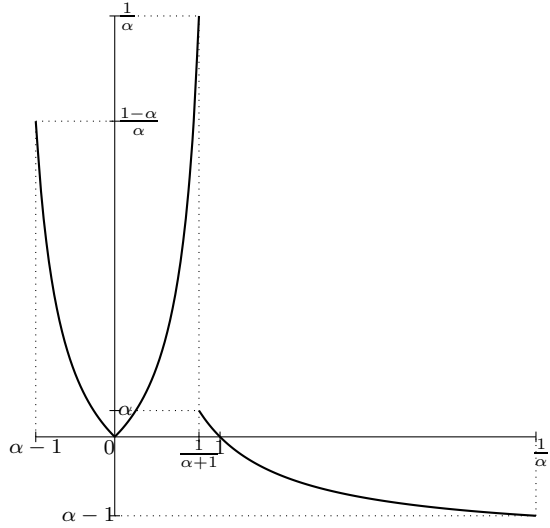
We will see in \S2 that the mediant convergents are induced from $F_{\alpha}$.  However, we should note that if $\varepsilon_{\alpha, n+1}(x) = -1$,
\[
\frac{1 \cdot p_{\alpha, n}(x) - p_{\alpha, n-1}(x)}{1 \cdot q_{\alpha, n}(x) - q_{\alpha, n-1}(x)}
=
\frac{(a_{\alpha, n}(x)-1)\cdot p_{\alpha, n-1}(x) + \varepsilon_{\alpha, n}(x) p_{\alpha, n-2}(x) }{(a_{\alpha, n}(x)-1) \cdot q_{\alpha, n-1}(x) + \varepsilon_{\alpha, n}(x)q_{\alpha, n-2}(x) },
\]
which means that we get the same rational number more than once as a mediant convergent.  To avoid such repetitions, Natsui in 2004 introduced in~\cite{Nats-1} another type of a Farey like map $F_{\alpha, \flat}$ for $\tfrac{1}{2}\leq\alpha < 1$, which is an induced transformation of $F_{\alpha}$, and was defined on $[\alpha -1, 1]$ by
\begin{equation} \label{eq:5}
F_{\alpha, \flat}(x) = \begin{cases}
- \frac{x}{1+x}, & \text{if $\alpha-1\leq x < 0$}; \\
\frac{x}{1-x}, & \text{if $0\leq x < \frac{1}{2}$}; \\
\frac{1 -2x}{x}, & \text{if $\frac{1}{2}\leq x\leq\frac{1}{1 + \alpha}$}; \\
\frac{1 -x}{x}, & \text{if $\frac{1}{1 + \alpha}< x\leq 1$}.
\end{cases}
\end{equation}
The definition of $F_{\alpha, \flat}$ as given in~\eqref{eq:5} does not work for the case $0<\alpha < \tfrac{1}{2}$, since the image of $[\alpha -1, 0)$ under $F_{\alpha, \flat}$ is not contained in $[\alpha -1 , 1]$. Indeed, $F_{\alpha, \flat}(\alpha -1) = \frac{1 - \alpha}{\alpha} > 1$ for $\alpha < 1/2$.  For this reason, we modify the above definition of $F_{\alpha, \flat}$ slightly; see~\eqref{eq:6} in \S~\ref{section2}. Both are induced transformations, but with a slightly different definition. In the sequel, we first show that $F_{\alpha}$ certainly induces the mediant convergents and is well-defined for $0 < \alpha < 1/2$.  Then we introduce a simple variant of $F_{\alpha, \flat}$.  We will show that dynamically these maps are isomorphic to the Farey map $F$ in the following sense.  In \S3, we construct a planer map $\hat{F}_{\alpha}$ which is the natural extension of $F_{\alpha}$ and then construct in \S4 the natural extension of $F_{\alpha, \flat}$
(denoted by $\hat{F}_{\alpha, \flat}$) as an induced map of $\hat{F}_{\alpha}$. Then we show that for $0 < \alpha < 1$, all $\hat{F}_{\alpha, \flat}$ are metrically isomorphic to $\hat{F}_{1}$. One of the points which we have to be careful is
that the first coordinates of planer maps of the ``mediant convergent maps $\hat{F}_{\alpha, \cdot}$, $0<\alpha<1$" are
not the ``mediant convergent maps $F_{\alpha, \cdot}$", though the first coordinate of the natural extension maps $\hat{G}_{\alpha}$ are exactly the $\alpha$-continued fraction maps $G_{\alpha}$.

In \S 5, we apply the idea of the planer maps to show some results on $\alpha$-continued fractions, which are already known for $\tfrac{1}{2} \leq \alpha  \le 1$ but not for $0 < \alpha < \tfrac{1}{2}$. We recall some results on normal numbers and on mixing properties of
$G_{\alpha}$; see~\cite{K-N} and~\cite{N-N-1}, respectively.   The first application of the $\alpha$-Farey map is a relation among normal numbers with respect to $\alpha$-continued fractions for different values of $\alpha$. In~\cite{K-N}, it was shown that the set of normal numbers with respect to $G_{\alpha}$ is the same for any $\tfrac{1}{2} \le \alpha \le 1$.   It is natural to ask whether we can extend the result to $0< \alpha \leq 1$.   However, the proof used in~\cite{K-N} does not work.   The main point is that the sequence of the principal convergents $\left( \tfrac{p_{\alpha, n}}{q_{\alpha, n}}:n \ge 1\right)$ is a subsequence of $\left( \tfrac{p_{n}}{q_{n}}:n \ge 1\right)$ for
$\tfrac{1}{2} \le \alpha \le 1$. This also holds for $\sqrt{2} - 1 \le \alpha < \tfrac{1}{2}$ but not anymore for
$0 < \alpha < \sqrt{2} - 1$.  Thus it is easy to follow the proof used in~\cite{K-N} for $\sqrt{2} - 1 \le \alpha < \tfrac{1}{2}$ but not possible for $\alpha$ below $\sqrt{2} - 1$.  At this point, we need the $\alpha$-mediant convergents to discuss normality. The second application is the following. In~\cite{N-N-1}, we show that the $\phi$-mixing property fails for a.e. $\alpha \in [\tfrac{1}{2}, 1]$ using the above normal number result. Indeed, the result implies that for a.e.\ $\alpha \in [\tfrac{1}{2}, 1]$, $\left\{ G_{\alpha}^{n}(\alpha - 1) : n \ge 1\right\}$ is dense in $\mathbb I_{\alpha}$. Then for any $\varepsilon > 0$, we can find $n \ge 1$ such that the size of either the interval $[\alpha - 1, G_{\alpha}^{n}(\alpha - 1)]$ or $[G_{\alpha}^{n}(\alpha - 1), \alpha)$ is less than $\varepsilon$.   The property called ``matching" plays an
important role there.  It was proved in \cite{N1981}, however it seems that nobody, not even the author of~\cite{N1981}, noticed the importance of this property until~\cite{N-N-2} appeared (after \cite{N-N-1}!). It is also easy to see the matching property for $\sqrt{2} - 1 \leq \alpha < \tfrac{1}{2}$ holds, but not easy for $\alpha$ below $\sqrt{2} - 1$.  After \cite{N-N-2} was published, in~\cite{C-T} the complete characterisation of the set of $\alpha$'s which have the matching property was given together with the proof of a conjecture from~\cite{N-N-2}. Actually, the matching property holds for almost all $\alpha \in (0, 1)$. Together with the result from \S 5.1, we show in \S 5.2 that $G_{\alpha}$ is not $\phi$-mixing for almost every $\alpha \in (0, 1)$. In \S 5 the construction of the natural extension $\hat{F}_{\alpha, \flat}$ of $F_{\alpha, \flat}$ as a planer map plays an important role.\smallskip\

In this paper, we change the notation in \cite{Nats-1} and \cite{Nats-2} to adjust for the names of Gauss and Farey:
$$
\begin{array}{ccc}
\text{\cite{Nats-1,Nats-2}} & {} & \text{this note} \\
T_{\alpha}   & \to & G_{\alpha} \\
G_{\alpha}   & \to & F_{\alpha} \\
F_{\alpha} & \to & F_{\alpha, \flat}
\end{array}
$$
%%%%%%%%%%%%%
%%%%%%%%%%%%%
\section{Basic properties of the $\alpha$-Farey map $F_{\alpha}$, $0 < \alpha < 1$}\label{section2}
First of all, note that there is a strong relation between the maps $G_{\alpha}$ from~\eqref{eq:3} and $F_{\alpha}$ from~\eqref{eq:4}. For any $\alpha \in (0, 1)$, we get $G_{\alpha}$ as a induced transformation of $F_{\alpha}$. This induced transformation is defined as follows.

For each $\alpha \in (0, 1)$ and $x\in \mathbb I_{\alpha} = [\alpha -1,\alpha)$, we put $j(x) = j_{\alpha}(x) = k$ if $x\ne 0$, $F_{\alpha}^{\ell}(x) \notin  (\frac{1}{1 + \alpha}, \frac{1}{\alpha}]$, $0 \le \ell < k$ and $F_{\alpha}^{k}(x) \in (\frac{1}{1 + \alpha}, \frac{1}{\alpha}]$ (note that from definition~\eqref{eq:4} of $F_{\alpha}$ we then have that $F_{\alpha}^{k+1}(x)\in \mathbb I_{\alpha}$, which is the domain of $G_{\alpha}$; see~\eqref{eq:3}), and $j(0)=0$ if $x = 0$ (noting $F_{\alpha}^{\ell}(0) = 0$ for any positive integer $\ell$). In case $\tfrac{\sqrt{5}-1}{2} < \alpha < 1$ we further define $j(x)=0$ whenever $x\in [\tfrac{1}{\alpha +1} ,\alpha )$; see also Remark~\ref{RemarksOnVariousAlphas}($i$). As usual, we set that $F_{\alpha}^{0}(x) = x$. Now the induced transformation $F_{\alpha, J}$ is defined as:
$$
F_{\alpha, J}(x) = F_{\alpha}^{j(x) + 1}(x),\quad \text{for $x \in \mathbb I_{\alpha}$}.
$$
The next proposition generalizes the result in~\cite{Nats-1}, where $\alpha$ was restricted to the interval $[\tfrac{1}{2}, 1]$.
%%%%%%%%%%%%%%%%%%
\begin{prop} \label{prop-1}
For any $0 < \alpha \le 1$, we have $G_{\alpha}(x) = F_{\alpha,J}(x)$
for any $x \in \mathbb I_{\alpha}$.
\end{prop}
%%%%%%%%%%%%%%%%%%%

\begin{proof}
Since $F_{\alpha}(0) = 0$, $F_{\alpha, J}(0) =0 = G_{\alpha}(0)$ is trivial.  Next we consider the case $x \in [\alpha -1, 0)$.  If $-\frac{1}{x} \in [(n-1) + \alpha, n + \alpha)$, then $F_{\alpha}(x) = (R\circ S_{1} \circ R{-})(x) \in (\frac{1}{(n-1) + \alpha},  \frac{1}{(n-2) + \alpha}]$.  Thus we get $F_{\alpha}^{n-1}(x) \in (\frac{1}{1 + \alpha}, \frac{1}{\alpha}]$ inductively, and $j(x) = n-1$ in this case.  We also see that $F_{\alpha}^{n-1}(x) = - \frac{x}{(n-1)x + 1}$ and then
$F_{\alpha, J}(x) = F_{\alpha}^{j(x) + 1}(x) =
F_{\alpha}\left( - \frac{x}{(n-1)x + 1}\right) = \left(- \frac{1}{x} - (n-1) \right)- 1 = - \frac{1}{x} - n = G_{\alpha}(x)$.  For $x \in (0, \alpha)$, the same proof holds since $F_{\alpha}(x) = (R\circ S_{1} \circ R)(x) \in (\frac{1}{(n-1) + \alpha}, \frac{1}{(n-2) + \alpha}]$ when $\frac{1}{x} \in [(n-1) + \alpha, n + \alpha)$. The rest of the proof is straightforward.
\end{proof}

\begin{rem}\label{RemarksOnVariousAlphas}{\rm
We gather some results on $j(x)$ for various values of $\alpha$.
\begin{itemize}
\item[($i$)]
$\frac{\sqrt{5} - 1}{2} < \alpha < 1$.\\
In this case $\frac{1}{1 + \alpha}< \alpha$ holds. It is for this reason we defined $j(x) = 0$ for $x\in [\tfrac{1}{\alpha +1} ,\alpha)$. On the other hand, $j(x)\geq 2$ for $\alpha -1 \le x < 0$.

\item[($ii$)]
$0 < \alpha \le \frac{\sqrt{5} - 1}{2}$.\\
In this case we have  $\frac{1}{1 + \alpha}\geq \alpha$ and $1 + \alpha < \frac{1}{1 - \alpha}$, which show that $j(x) = 0$ \emph{only} for $x=0$, and $j(x) = 1$ for $x \in [\alpha - 1, - \tfrac{1}{1+ \alpha}) \cup [ \tfrac{1}{2 + \alpha}, \alpha]$.

\item[($iii$)]
$0 < \alpha < \sqrt{2} - 1$.\\
We see that $j(x)\geq 2$ for $x \in (0, \alpha)$.\hfill $\triangle$\smallskip\
\end{itemize}
}
\end{rem}

Setting $A^{-} = \begin{pmatrix} -1 & 0 \\ 1 & 1 \end{pmatrix}$, $A^{+} = \begin{pmatrix} 1 & 0 \\ -1 & 1 \end{pmatrix}$ and
$A^{R} = \begin{pmatrix} -1 & 1 \\ 1 & 0 \end{pmatrix}$, in view of~\eqref{eq:2} and~\eqref{eq:4} we can write $F_{\alpha}$ as:
\[
F_{\alpha}(x) = \begin{cases}
A^{-}x, & \text{if $x \in [\alpha - 1, 0)$}; \\
A^{+}x, & \text{if $x \in [0, \frac{1}{1 + \alpha}]$};  \\
A^{R}x, & \text{if $x\in (\frac{1}{1 + \alpha}, \frac{1}{\alpha}]$}.
\end{cases}
\]
Define
\[
A_{n}(x) =  A(F_{\alpha}^{n-1}(x) )=  \begin{cases}
\begin{pmatrix} -1 & 0 \\ 1 & 1 \end{pmatrix} = (A^{-})^{-1}, & \text{if $F_{\alpha}^{n-1} (x) \in [\alpha - 1, 0)$}; \\
\begin{pmatrix} 1 & 0 \\ 0 & 1 \end{pmatrix}, & \text{if $F_{\alpha}^{n-1}(x) = 0$};\\
\begin{pmatrix} 1 & 0 \\ 1 & 1 \end{pmatrix} = (A^{+})^{-1}, & \text{if $F_{\alpha}^{n-1}(x) \in (0, \frac{1}{1 + \alpha}]$};  \\
\begin{pmatrix} 0 & 1 \\ 1 & 1 \end{pmatrix} = (A^{R})^{-1}, & \text{if $F_{\alpha}^{n-1}(x)\in (\frac{1}{1 + \alpha}, \frac{1}{\alpha}]$},
\end{cases}
\]
for $x \in \mathbb I_{\alpha}$ and $n \ge 1$.  We identify $x$ with
$(A_{1}(x), A_{2}(x), \ldots , A_{n}(x), \ldots)$. We will show that
\[
\lim_{n \to \infty} \left(A_{1}(x) A_{2}(x) \cdot \cdots A_{n}(x)\right)(-\infty) = x .
\]
Put
$\begin{pmatrix}s_{n} & u_{n} \\ t_{n} & v_{n} \end{pmatrix}
= A_{1}(x)  A_{2}(x)  \cdots  A_{n}(x)$ with
$\begin{pmatrix}s_{0} & u_{0} \\ t_{0} & v_{0} \end{pmatrix} =
\begin{pmatrix}1 & 0 \\ 0 &1 \end{pmatrix} $.  Suppose that
\[
x = \confrac{\varepsilon_{\alpha, 1}(x)}{a_{\alpha, 1}(x)} +
\confrac{\varepsilon_{\alpha, 2}(x)}{a_{\alpha, 2}(x)} + \cdots +
\confrac{\varepsilon_{\alpha, n}(x)}{a_{\alpha, n}(x)}  + \cdots
\]
is the $\alpha$-continued fraction expansion of $x \in \mathbb I_{\alpha}$.
Recall that $\varepsilon_{\alpha, n}(x)= \text{sgn}(G_{\alpha}^{n-1}(x))$ and
$a_{\alpha, n}(x) = \lfloor\frac{1}{|G_{\alpha}^{n-1}(x)|}+ 1 -\alpha \rfloor$ for $G_{\alpha}^{n-1}(x) \ne 0$.  Then, from Proposition~\ref{prop-1}, it is easy to see that $(A_{1}(x), A_{2}(x), \ldots , A_{n}(x), \ldots )$ is of the form
\begin{equation} \label{eq-add-1}
(A^{\pm}, \underbrace{A^{+}, \ldots , A^{+}}_{a_{\alpha, 1}(x) -2}, A^{R},
 A^{\pm}, \underbrace{A^{+}, \ldots , A^{+}}_{a_{\alpha, 2}(x) -2}, A^{R}, \ldots ) ,
\end{equation}
unless $A_{m} = \begin{pmatrix} 1 & 0 \\ 0 & 1 \end{pmatrix}$ appears in~\eqref{eq-add-1} for some $m\geq 1$ (which happens when $x\in\Q$). Here $A^{\pm} = A^{-}$ or $A^{+}$ according to $\varepsilon_{\alpha, n}(x) = -1$ or $+1$, respectively. If $a_{\alpha, n}(x) = 1$, then we read $a_{\alpha, n} -2 = 0$ and delete $A^{\pm}$ before $A^{+}$.
More precisely,
\[
\begin{array}{lcl}
A_{\sum_{j=1}^{n} a_{\alpha, j}(x)}(x) &= & A^{R}; \\
A_{\sum_{j =1}^{n} a_{\alpha, j}(x) + 1}(x) &= &\left\{\begin{array}{ccl}
A^{-}, & \text{if} & \varepsilon_{\alpha, n}(x) = -1; \\
A^{+}, & \text{if} & \varepsilon_{\alpha, n}(x) = +1 \quad \text{and} \quad a_{\alpha, n+1}(x) \ge 2;
\end{array} \right.
 \\
A_{\sum_{j=1}^{n} a_{\alpha, j}(x) + \ell}(x) & =& A^{+} \qquad \text{if} \,\,\,\, 2 \le \ell < a_{\alpha, n+1},
\end{array}
\]
with $\sum_{j = 1}^{0} a_{\alpha, j}(x) = 0$. As usual, we have for the $G_{\alpha}$-convergents of $x$ that:
\[
\begin{pmatrix} 0 & \varepsilon_{\alpha, 1}(x) \\ 1 & a_{\alpha, 1}(x) \end{pmatrix}
\begin{pmatrix} 0 & \varepsilon_{\alpha, 2}(x)\\ 1 & a_{\alpha, 2}(x) \end{pmatrix}
\cdots
\begin{pmatrix} 0 & \varepsilon_{\alpha, n}(x) \\ 1 & a_{\alpha, n}(x) \end{pmatrix}
= \begin{pmatrix} p_{\alpha, n-1}(x) & p_{\alpha, n}(x) \\q_{\alpha, n-1}(x)
& q_{\alpha, n}(x)
\end{pmatrix} .
\]
We have the following result.
\begin{lem} \label{lem-1}
For $k,n\in\N$, let\, $\Pi_k(x) := A_{1}(x) A_{2}(x) \cdots A_{k}(x)$, and $S_{n}(x) = \sum_{j=1}^{n} a_{\alpha, j}(x)$. Then, if $k = S_{n}(x)$,
$$
\Pi_k(x) = \begin{pmatrix} p_{\alpha, n-1}(x) & p_{\alpha, n}(x) \\ q_{\alpha, n-1}(x)
& q_{\alpha, n}(x) \end{pmatrix}.
$$
Furthermore, if $\ell \geq 1$ and $k = S_{n}(x)  + \ell < S_{n+1}(x)$, we have
$$
\Pi_k(x) = \begin{pmatrix} \ell p_{\alpha, n}(x)+\varepsilon_{\alpha, n+1}(x)p_{\alpha, n-1}(x) & p_{\alpha, n-1}(x) \\
\ell q_{\alpha, n}(x)+\varepsilon_{\alpha, n+1}(x)q_{\alpha, n-1}(x) & q_{\alpha, n-1}(x) \end{pmatrix}.
$$
\end{lem}

\begin{proof}
The assertion of this lemma follows from an easy induction and is essentially due to the fact that $G_{\alpha}$ is an induced transformation of $F_{\alpha}$.
\end{proof}

With this result in mind, and analogously to the regular case (which is $\alpha = 1$), we call
$$
(A_{1}(x) A_{2}(x) \cdots A_{k}(x))(-\infty) = \frac{\ell p_{\alpha, n}(x) + \varepsilon_{\alpha,n+1}p_{\alpha, n-1}(x)}{\ell q_{\alpha, n}(x)
+ \varepsilon_{\alpha,n+1}q_{\alpha, n-1}(x)}
$$
the ($(n+1, \ell)$-th) $\alpha$-mediant convergent of $x$ for $\ell > 0$; see also~\eqref{eq:1} for the case $\alpha = 1$.\medskip\

\begin{rem}\label{Remark1}{\rm For $0 < \alpha \le \frac{\sqrt{5}-1}{2}$, $a_{\alpha, n}(x) \ge 2$
for any $x \in \mathbb I_{\alpha}$ and $n \ge 1$. {\phantom{XXX}}\hfill $\triangle$}
\end{rem} \smallskip\

From Lemma~\ref{lem-1}, the convergence of the mediant convergents follows:
\begin{prop}\label{prop-2}
We have
\[
\lim_{k \to \infty} (\Pi_{k}(x))(-\infty) = x .
\]
\end{prop}

\begin{proof} If $x$ is rational, then the assertion follows easily.  So we estimate $|x  - (\Pi_{k}(x))(-\infty)|$ for an irrational $x$.
We note that $(q_{\alpha, n}(x) :n \ge 1)$ is strictly increasing for any $x \in \mathbb I_{\alpha} \setminus \Q$, which follows
from the fact that $a_{\alpha, n}(x) \ge 2$ if $\varepsilon_{\alpha, n}(x) = -1$ (for any $\alpha,\,\, 0 < \alpha \le 1$).
This implies $\lim_{n \to \infty} q_{\alpha, n}(x) = \infty$ if $x$ is irrational. As for the RCF, see e.g.~\cite{DK2002} or~(1.1.14) in~\cite{IK}, $x$ can be written as:
\begin{equation}\label{xAndItsFuture}
\frac{(\ell p_{\alpha, n}(x) \pm p_{\alpha, n-1}(x))F_{\alpha}^{k}(x) +
p_{\alpha, n-1}(x)}{(\ell q_{\alpha, n}(x) \pm q_{\alpha, n-1}(x))F_{\alpha}^{k}(x)
+ q_{\alpha, n-1}(x)} \text{ or }\,\, \frac{p_{\alpha, n-1}(x)G_{\alpha}^{n}(x) + p_{\alpha, n}(x)}
{q_{\alpha, n-1}(x)G_{\alpha}^{n}(x) + q_{\alpha, n}(x)} .
\end{equation}
The estimate of the latter is easy since $(\Pi_{k}(x))(-\infty) = \frac{p_{\alpha, n-1}(x)}{q_{\alpha, n-1}(x)}$, $|G_{\alpha}^{n}(x)|$ $<
\max(\alpha, 1 - \alpha) < 1$, and $|p_{\alpha, n-1}(x)q_{\alpha, n}(x) - p_{\alpha, n}(x)q_{\alpha, n-1}(x)| = 1$.
Anyway, it is the convergence estimate of the $\alpha$-continued fraction expansion of $x$.   In the former case,  we see that $\left| x - (\Pi_{k}(x))(-\infty)\right|$ is equal to:
$$
\left| \frac{(\ell p_{\alpha, n}(x) \pm p_{\alpha, n-1}(x))F_{\alpha}^{k}(x) +
p_{\alpha, n-1}(x)} {(\ell q_{\alpha, n}(x) \pm q_{\alpha, n-1}(x))F_{\alpha}^{k}(x)
+ q_{\alpha, n-1}(x)} - \frac{\ell p_{\alpha, n}(x) \pm p_{\alpha, n-1}(x)}
{\ell q_{\alpha, n}(x) \pm q_{\alpha, n-1}(x)} \right| ,
$$
with $k = \sum_{j=1}^{n} a_{\alpha, j}(x) + \ell$. This can be estimated as
\begin{align*}
{}&\left|
\frac{\ell}{\left((\ell q_{\alpha, n}(x) \pm q_{\alpha, n-1}(x))F_{\alpha}^{k}(x) +
q_{\alpha, n-1}(x)\right) \left(\ell q_{\alpha, n}(x) \pm q_{\alpha, n-1}(x)\right)}
\right|\\
=& \left| \frac{1}{\left(q_{\alpha, n}(x) \pm \frac{q_{\alpha, n-1}(x)}{\ell}\right)
\left((\ell q_{\alpha, n}(x) \pm q_{\alpha, n-1}(x))F_{\alpha}^{k}(x) +
q_{\alpha, n-1}(x)\right)}
\right|\\
 < & \,\, \frac{1}{q_{\alpha, n-1}(x)} \quad \to \quad 0
\end{align*}
Here we used the fact $F_{\alpha}^{k} (x)\ge 0$ for $k$ not of the form
$\sum_{j = 1}^{n} a_{\alpha, j}(x)$.
\end{proof}
%%%%%%%%%%%%%%%%%%%%%%
%%%%%%%%%%%%%%%%%%%%%%

As mentioned in the introduction, the $(n+1, a_{\alpha, n+1}(x)-1)$-th convergent is the same as the $(n+2, 1)$-th convergent if $\varepsilon_{\alpha, n+2}(x) = -1$, i.e.,
\[
\frac{(a_{\alpha, n+1}(x) -1)p_{\alpha, n}(x) +\varepsilon_{\alpha, n+1}(x) p_{\alpha, n-1}(x)}{(a_{\alpha, n+1}(x) -1)q_{\alpha, n}(x) +\varepsilon_{\alpha, n+1}(x) q_{\alpha, n-1}(x)}
 =
\frac{p_{\alpha, n+1}(x) - p_{\alpha, n}(x)}
{q_{\alpha, n+1}(x) - q_{\alpha, n}(x)}  .
\]
In this sense,  the map $F_{\alpha}$ makes a duplication if $\varepsilon_{\alpha, n}(x) = -1$. This duplication is $k$-fold if $(\varepsilon_{\alpha, n+ \ell}(x), a_{\alpha, n+\ell}(x)) = ( -1, 2)$ for $1\leq\ell\leq k$, i.e., 
\begin{eqnarray*}
\frac{p_{\alpha, n}(x) - p_{\alpha, n-1}(x)}{q_{\alpha, n}(x) - q_{\alpha, n-1}(x)} &=& \frac{p_{\alpha, n+1}(x) - p_{\alpha, n}(x)}{q_{\alpha, n+1}(x) - q_{\alpha, n}(x)}\,\, =\,\,\cdots \\
&=& \frac{p_{\alpha, n+k}(x) - p_{\alpha, n+k-1}(x)}{q_{\alpha, n+k}(x) - q_{\alpha, n+k-1}(x)} .
\end{eqnarray*} 

We avoid this duplication using a suitable induced transformation. First, let us recall the definition of $F_{\alpha, \flat}$ for $\frac{1}{2}\leq\alpha < 1$; see~\eqref{eq:5}.  One can see that this map skips the $(n, a_{n+1}(x)-1)$-th mediant convergent of $x \in \mathbb I_{\alpha}$ with $\varepsilon_{n+1}(x) = -1$. An important observation in~\cite{Nats-1} is that for $\tfrac{1}{2}\leq\alpha < 1$ we have that $-\frac{x}{1+ x} < 1$ for any $\alpha -1\leq x < 0$. This does not apply anymore when $0 < \alpha < \frac{1}{2}$, as the definition of $F_{\alpha, \flat}$ should be on the interval $[\alpha -1, 1]$. To achieve this we ``speed up'' $F_{\alpha}$ and modify the definition of $F_{\alpha, \flat}$ as follows:
\[
F_{\alpha, \flat}(x) = F_{\alpha}^{K(x)}(x),
\]
with $K(x) = \min \{ k\geq 1 : F_{\alpha}^k(x) \in [\alpha - 1 , 1]$.

For $\alpha - 1\leq x < -\tfrac{1}{2}$, $F_{\alpha}(x) = - \frac{x}{1 + x} > 1$ (see also~\eqref{eq:4}), so $F_{\alpha}^{2}(x) \in \mathbb I_{\alpha}$.  Thus $K(x) = 2$  and we have that $F_{\alpha}^{K(x)}(x) =\frac{1 + 2x}{x}$ in this case.
For $x\in [-\frac{1}{2}, \frac{1}{2})$, one can easily see that $F_{\alpha}(x) \in
[0, 1]$ and the same holds also for $x \in [\frac{1}{1 + \alpha}, 1]$.  For $x \in [\frac{1}{2}, \frac{1}{1 + \alpha})$, we find that $K(x) = 2$ and
$F_{\alpha, \flat}(x) = \frac{1 -2x}{x}$.
Consequently, our new definition of $F_{\alpha, \flat}$ is the following:
\begin{equation} \label{eq:6}
F_{\alpha, \flat}(x) = \begin{cases}
F_{\alpha}^2(x) = - \frac{1 + 2x}{x},   & \text{if $\alpha-1\leq x < - \frac{1}{2}$};\\
F_{\alpha}(x) = - \frac{x}{1+x}, & \text{if $-\frac{1}{2}\leq x < 0$}; \\
F_{\alpha}(x) = \frac{x}{1-x}, & \text{if $0\leq x < \frac{1}{2}$}; \\
F_{\alpha}^2(x) = \frac{1-2x}{x}, & \text{if $\frac{1}{2}\leq x  < \frac{1}{1 + \alpha}$}: \\
F_{\alpha}(x) = \frac{1-x}{x}, & \text{if $\frac{1}{1 + \alpha}\leq x\leq 1$}.
\end{cases}
\end{equation}
%%%%%%%%%%%%%%%%
%%%%%%%%%%%%%%%%
Clearly from~\eqref{eq:6} we have that for every $x\in [\alpha -1,1]$ the sequence $(F_{\alpha, \flat}^k(x))_{k\geq 0}$, which is the orbit of $x$ under $F_{\alpha, \flat}$, is a subsequence of the sequence $(F_{\alpha}^n(x))_{n\geq 0}$ (the orbit of $x$ under $F_{\alpha}$). But then for every $x\in [\alpha -1,1]$ fixed there exists a (unique) monotonically increasing function $\hat{k}:\N\cup \{ 0\}\to\N\cup\{ 0\}$, such that $F_{\alpha, \flat}^{k}(x) = F_{\alpha}^{\hat{k}(k)}(x)$. Setting $\hat{k}=\hat{k}(k)$, for $k=0,1,\dots$, we put
\[
\Pi_{\flat, k}(x) = \Pi_{\hat{k}}(x).
\]
From this definition, it is easy to derive the following.
%%%%%%%%%%%%%%%%%%%
%%%%%%%%%%%%%%%%%%%
\begin{prop}\label{prop-3}
For any $k\ge 1$, whenever $\varepsilon_{\alpha,n+1}(x)=-1$, the $(n+1, a_{n+1}-1)$-th mediant convergent does not appear  for any $n \ge 1$ in $(\Pi_{\flat, k}(x)(-\infty): k\ge 1)$ and all other mediant convergents and all principal convergents of $x$ appear in it.
\end{prop}
%%%%%%%%%%%%%%%%%%%
%%%%%%%%%%%%%%%%%%%
Another possibility is by skipping $p_{n+1}(x) - p_{n}(x)$ and $q_{n+1}(x) - q_{n}(x)$ instead of $(a_{n+1}(x) -1))p_{n}(x) + \varepsilon_{n+1}(x) p_{n-1}(x)$ and $(a_{n+1}(x) -1))q_{n}(x) + \varepsilon_{n+1}(x)q_{n-1}(x)$ if $\varepsilon_{n+1}(x) = -1$. This can be done by the jump transformation $F_{\alpha, \sharp}$, defined as follows:
\begin{equation}  \label{eq:7}
F_{\alpha, \sharp}(x) = \begin{cases}
F_{\alpha}^{2}(x), & \text{if $x <0$};\\
F_{\alpha}(x), & \text{if $x\geq 0$}.
\end{cases}
\end{equation}
Note that the map $F_{\alpha, \sharp}$ from~\eqref{eq:7} is explicitly given by:
\begin{equation} \label{eq:8}
F_{\alpha, \sharp}(x) = \begin{cases}
- \frac{x}{1+2x}, & \text{if $\alpha-1\leq x < 0$}; \\
\frac{x}{1-x}, & \text{if $0\le x < \frac{1}{1+\alpha}$}; \\
\frac{1 -x}{x}, & \text{if $\frac{1}{1 + \alpha} \le x \le \frac{1}{\alpha}$}.
\end{cases}
\end{equation}
This is well-defined for any $0< \alpha < 1$.  Indeed, the map $F_{\alpha, \sharp}$ from~\eqref{eq:8} skips $F_{\alpha}^{k+1}(x)$ if $F_{\alpha}^{k}(x) <0$, which implies that there exists an $n\geq 1$ such that $G_{\alpha}^{n}(x) =F_{\alpha}^{k}(x)$ and $\varepsilon_{n}(x) = -1$.  Thus we see that $\frac{p_{n}(x) - p_{n-1}(x)}{q_{n}(x) - q_{n-1}(x)}$ has been skipped in the sequence of the mediant convergents. In this note, we do not further consider this
map $F_{\alpha, \sharp}$ since the discussion is almost the same as that of $F_{\alpha, \flat}$.

Now we consider $\frac{s_{k}(x)}{t_{k}(x)} = \Pi_{k}(x)(-\infty)$ with $s_{k}(x), t_{k}(x) \in \mathbb Z$, coprime, $t_{k} > 0$.
From this and~\eqref{xAndItsFuture} we derive that
\[
t_{k}^{2}(x) \left| x - \frac{s_{k}(x)}{t_{k}(x)}\right| = \left|F_{\alpha}^{k}(x) - \Pi_{k}^{-1}(-\infty)\right|^{-1},
\]
for $x \in [\alpha-1, \frac{1}{\alpha}]$ and $n\geq 1$. Note that $F_{\alpha}^{k}(x)$ can be interpreted as \emph{the future of $x$ at time $k$}, while $\Pi_{k}^{-1}(-\infty)$ is like \emph{the past of $x$ at time $k$}; see also Chapter~4 in~\cite{DK2002}.  For this reason, it is interesting to find the closure of the set
$$
\left\{\left(F_{\alpha}^{k}(x), \Pi_{k}^{-1}(x)(-\infty)\right) : \,\, x \in [\alpha-1, \frac{1}{\alpha}], \,\, k >0 \right\}.
$$
This leads us to consider the following maps:
\begin{equation} \label{eq:9}
\hat{F}_{\alpha}(x, y) = \begin{cases}
\left( - \frac{x}{1+x}, - \frac{y}{1+y}\right) , & \text{if $\alpha -1\leq x < 0$}; \\
\left( \frac{x}{1-x}, \frac{y}{1-y}\right) , & \text{if $0\le x < \frac{1}{1 + \alpha}$}; \\
\left( \frac{1 -x}{x}, \frac{1 - y}{y}\right) , & \text{if $\frac{1}{1 + \alpha}\leq x < \frac{1}{\alpha}$},
\end{cases}
\end{equation}
and
\begin{equation} \label{eq:10}
\hat{F}_{\alpha, \flat}(x, y) = \begin{cases}
\left(- \frac{1+2x}{x},  - \frac{1+2y}{y} \right) , & \text{if $\alpha - 1\leq x < -\frac{1}{2}$}; \\
\left(- \frac{x}{1+x},  - \frac{y}{1+y}\right) , & \text{if $-\frac{1}{2}\leq x < 0$}; \\
\left(\frac{x}{1-x},  \frac{y}{1-y}\right) ,  & \text{if $0\leq x < \frac{1}{2}$}; \\
\left(\frac{1 -2x}{x}, \frac{1 -2y}{y}\right) ,  & \text{if $\frac{1}{2}\leq x \leq \frac{1}{1 + \alpha}$}; \\
\left(\frac{1 -x}{x}, \frac{1 -y}{y}\right) ,  & \text{if $\frac{1}{1 + \alpha}< x \leq 1$},
\end{cases}
\end{equation}
where $(x, y)$ is in a `reasonable domain' of the definition of each map, respectively. The question is to find this `reasonable domain' for each case. This will be done in Theorem~\ref{thm-1} for $\hat{F}_{\alpha}$ and in Theorem~\ref{thm-2} for $\hat{F}_{\alpha, \flat}$.  For example, for $\hat{F}_{\alpha}$ the domain will be the closure of
\[
\left\{\left(F_{\alpha}^{k}(x),  \left(A_{1}(x) \cdots A_{k}(x)\right)^{-1}(-\infty)\right): \,\, x \in [\alpha-1, \frac{1}{\alpha}], \,\, k >0 \right\}
\]
so that $\hat{F}_{\alpha}$ is bijective except for a set of Lebesgue measure $0$. From this point of view, $\hat{F}_{\alpha}$ is the planar
representation of the natural extension of $F_{\alpha}$ in the sense of Ergodic theory.  Another point of view is that the characterization of quadratic surds by the periodicity of the map.  Indeed, it is easy to see that $x \in (0, 1)$ is strictly periodic
by the iteration of $F$ if and only if it is a quadratic surd and its algebraic conjugate is negative.  We can characterise the set of quadratic surds in a similar way with
$\hat{G}^{*}_{\alpha}$, see the next section, and also $\hat{F}_{\alpha}$.
We can apply the above to the construction of the natural extension of $F_{\alpha, \flat}$.  Indeed, it is obtained as an induced transformation of
$\hat{F}_{\alpha}$.  In the next section, we give a direct construction of the natural extension of $\hat{F}_{\alpha}$ as a tower of the natural extension $\hat{G}_{\alpha}^{*}$ of $G_{\alpha}$.
%%%%%%%%%%%%%
\section{The natural extension of $F_{\alpha}$ for $0< \alpha < \frac{1}{2}$}\label{Natural extension of Falpha}
As the case $\tfrac{1}{2}\leq\alpha\leq 1$ was discussed in~\cite{Nats-1,Nats-2}, in the rest of this paper we will focus on the case $0 < \alpha < \tfrac{1}{2}$.  We give some figures in the case of $\alpha = \sqrt{2} - 1$ for better understanding of the construction. We selected this value of $\alpha$ as an example as this is historically the first ``more difficult'' case; for $\alpha\in (\sqrt{2}-1,1]$ the natural extensions are simply connected regions which are the union of finitely many overlapping rectangles, while for $\alpha = \sqrt{2}-1$ the natural extension consists of two disjoint rectangles; see~\cite{L-M, M-C-M, dJK}. In~\cite{dJK} it is shown that for $\alpha\in\left( \frac{\sqrt{10}-3}{2},\sqrt{2}-1\right)$ there is a countably infinite number of disjoint connected regions. For $0<\alpha < \sqrt{2}-1$ it is not so easy to describe $\Omega_{\alpha}$ explicitly; see the discussion at the end of Section~\ref{section2}, and also~\cite{K-S-S, L-M, M-C-M, dJK}.\smallskip\

We start with the domain $\Omega_{\alpha}$ from~\cite{K-S-S} and the natural extension map $\hat{G}_{\alpha}:\Omega_{\alpha}\to\Omega_{\alpha}$, given by
\[
(x, y) \mapsto \begin{cases}
\left(- \frac{1}{x} - b, \frac{1}{- y + b}\right) , & \text{for $x < 0$}; \\
\left( \frac{1}{x} - b, \frac{1}{ y + b}\right) , & \text{for $x > 0$},
\end{cases}
\]
for $(x, y) \in \Omega_{\alpha}$.  Next we change $y$ to $-\frac{1}{y}$, i.e., we consider
\[
\Omega_{\alpha}^{*}  = \left\{(x, y) : \Big( x, - \frac{1}{y}\Big) \in \Omega_{\alpha}\right\} ;
\]
see~Figure~\ref{fig-1}, and the map $\hat{G}_{\alpha}^{*}: \Omega_{\alpha}^{*}\to\Omega_{\alpha}^{*}$, defined by:
\[
(x, y) \mapsto \begin{cases}
\left(- \frac{1}{x} - b, - \frac{1}{y} -b\right) , & \text{for $x < 0$}; \\
\left( \frac{1}{x} - b, \frac{1}{ y} - b\right), & \text{for $x > 0$},
\end{cases},
\]
where $b = \lfloor \left|\frac{1}{x}\right| + \alpha -1 \rfloor$; this gives another version of the natural extension, with which we work with for the rest of the paper. Recall from~\cite{K-S-S} that $\hat{G}_{\alpha}:\Omega_{\alpha}\to\Omega_{\alpha}$ (and therefore $\hat{G}_{\alpha}^{*}$) is bijective except for a set of Lebesgue measure $0$. The reason to move from $\Omega_{\alpha}$ to $\Omega_{\alpha}^{*}$ is that the first and second coordinate maps of $G_{\alpha}$ are similar. This allows for a more unified treatment. Also note that $\Omega_{\alpha}\subset [\alpha -1,\alpha]\times [0,1]$.

%%%%%%%%%%%%%%%%%%%%%%
%%%%%%%%%%%%%%%%%%%%%%
%%%%%%%%%%%%%
%%%%%%%%%%%%%
%%%%%%%%%%%%%%%%%%%%%
%%%%%%%%%%%%%%%%%%%%%
%%%%%%%%%%%%%
%%%%%%%%%%%%%
\begin{figure}[htb]
\begin{center}
\begin{tikzpicture}[scale=0.5]
\draw[thick] (-3, -4) -- (20, -4);
\node at (-2.4, 0.5){\tiny$\sqrt{2} - 2$};
\node at (-1.16, 1){\tiny$\frac{\sqrt{2} - 2}{2}$};
\node at (-0.3, 0.5){\tiny$0$}; \node at (1.64, 1){\tiny$\sqrt{2} - 1$};
\node at (2.8, 0.5){\tiny$\frac{\sqrt{2}}{ 2}$}; \node at (4, 1){\tiny$1$};
\node at (5.6, 0.5){\tiny$\sqrt{2} $};  \node at (6.8, 1){\tiny$\frac{\sqrt{2} + 2}{2}$};
\node at (9.6, 0.5){\tiny$\sqrt{2}+1 $};  \node at (13.6, 1){\tiny$\sqrt{2}+2 $};
\node at (-2.4, -4){$\cdot$};
\node at (-1.16, -4){$\cdot$};
\node at (1.64, -4){$\cdot$};
\node at (2.8, -4){$\cdot$}; \node at (4, -4){$\cdot$};
\node at (5.6, -4){$\cdot$};  \node at (6.8, -4){$\cdot$};
\node at (9.6, -4){$\cdot$};  \node at (13.6, -4){$\cdot$};

\draw[thick] (0, 4) -- (0, -9);

\fill[blue!50] (6.8, -4) -- (20, -4)--(20, -2.45)--(6.8, -2.45)--(6.8, -4);
\fill[blue!50] (13.6, -2) -- (20, -2) -- (20, -1.54) -- (13.6, -1.54) -- (13.6, -2);
\fill[red!50] (9.6, -4) -- (20, -4)--(20, -5.54)--(9.6, -5.54)--(9.6, -4);
\fill[red!50] (9.6, -6) -- (20, -6)--(20, -6.46)--(9.6, -6.46)--(9.6, -6);
\draw (6.8, -1) -- (6.8, -7);
\draw (9.6, -1) -- (9.6, -7);
\draw (13.6, -1) -- (13.6, -7);
\draw (17.6, -1) -- (17.6, -7);

\node at (8, -1){\small$\hat{\Omega}_{\alpha, 2, -}$};
\node at (12, -1){\small$\hat{\Omega}_{\alpha, 3, -}$};
\node at (16, -1){\small$\hat{\Omega}_{\alpha, 4, -}$};

\node at (7.3, -7){\small$\hat{\Omega}_{\alpha, 2, +} = \emptyset$};
\node at (12, -7){\small$\hat{\Omega}_{\alpha, 3, +}$};
\node at (16, -7){\small$\hat{\Omega}_{\alpha, 4, +}$};

\node at (-0.5, -8) {\small$-1$};  \node at (-0.3, -4.3) {\small$0$};
\node at (0, -4){$\cdot$};\node at (0,-8){$\Huge\cdot$};
\end{tikzpicture}
\end{center}
\caption{$\hat{\Omega}_{\alpha,k,\pm}$ for $\alpha = \sqrt{2} - 1$}
\label{fig-2}
\end{figure}
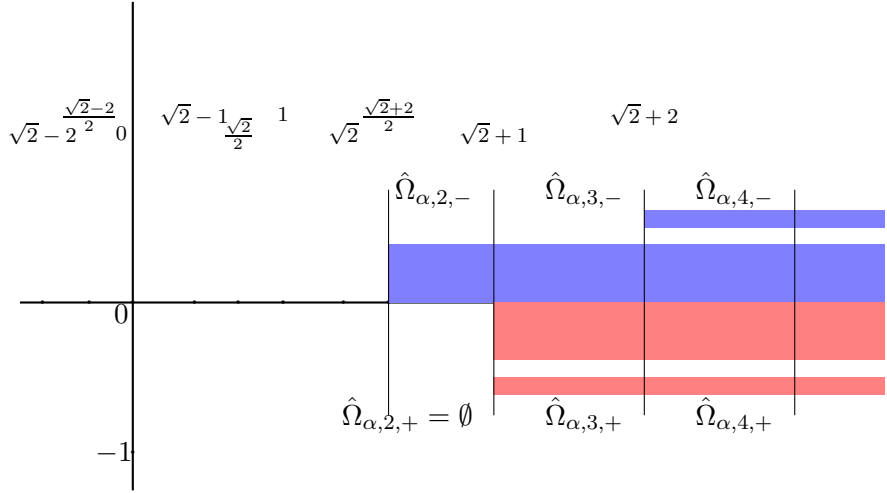
%%%%%%%%%%%%%%%%%%%%%
\begin{figure}
\begin{center}
\begin{tikzpicture}[scale=0.8]

\draw[thick] (-3, 0) -- (5, 0);

\node at (-2.4, -0.5){\small$\sqrt{2} - 2$};
\node at (-1.16, -0.5){\small$\frac{\sqrt{2} - 2}{2}$};
\node at (1.64, -0.5){\small$\sqrt{2} - 1$};

\fill[blue!50] (-1.2, -8) -- (1.6, -8)--(1.6, -6.45)--(-1.2, -6.45)--(-1.2, -8);

\fill[blue!50] (-2.4, -12) -- (1.6, -12)--(1.6, -10.45)--(-2.4, -10.45)--(-2.4, -12);

\fill[red!50] (-2.4, -12) -- (1.6, -12)--(1.6, -13.54)--(-2.4, -13.54)--(-2.4, -12);
\fill[red!50] (-2.4, -14) -- (1.6, -14)--(1.6, -14.46)--(-2.4, -14.46)--(-2.4, -14);

\fill[blue!50] (-2.4, -16) -- (1.6, -16)--(1.6, -14.45)--(-2.4, -14.45)--(-2.4, -16);
\fill[blue!50] (-2.4, -14) -- (1.6, -14) -- (1.6, -13.54) -- (-2.4, -13.54) -- (-2.4, -14);
\fill[red!50] (-2.4, -16) -- (1.6, -16)--(1.6, -17.54)--(-2.4, -17.54)--(-2.4, -16);
\fill[red!50] (-2.4, -18) -- (1.6, -18)--(1.6, -18.46)--(-2.4, -18.46)--(-2.4, -18);

\fill[blue!50] (-2.4, -20) -- (1.6, -20)--(1.6, -18.45)--(-2.4, -18.45)--(-2.4, -20);
\fill[blue!50] (-2.4, -18) -- (1.6, -18) -- (1.6, -17.54) -- (-2.4, -17.54) -- (-2.4, -18);

\node at (3.5, -7){\small$\hat{\Omega}_{\alpha, 2, -} - (2, 2)$};

\node at (3.5, -11){\small$\hat{\Omega}_{\alpha, 3, -} - (3, 3)$};

\node at (3.7, -13){\small$\hat{\Omega}_{\alpha, 3, +} - (3, 3)$};

\node at (3.5, -15){\small$\hat{\Omega}_{\alpha, 4, -} - (4, 4)$};

\node at (3.7, -17){\small$\hat{\Omega}_{\alpha, 4, +} - (4, 4)$};

\node at (3.5, -19){\small$\hat{\Omega}_{\alpha, 5, -} - (5, 5)$};

\draw[thick] (0, 1) -- (0, -21);

\node at (-3.5, -8) {\small$-2$};  \node at (-0.3, -0.3) {\small$0$};
\node at (-3.5, -4){\small{$-1$}};
\node at (-3.5, -6.45){\small{$-\frac{\sqrt{5} + 1}{2}$}};
\node at (-3.5, -10.54){\small{$-\frac{\sqrt{5}+ 3}{2}$}};

\node at (0, -4){$\cdot$};\node at (0,-8){$\Huge\cdot$};
\end{tikzpicture}
\end{center}
\caption{$\Omega_{\alpha}^*$ for $\alpha = \sqrt{2} - 1$}
\label{fig-1}
\end{figure}
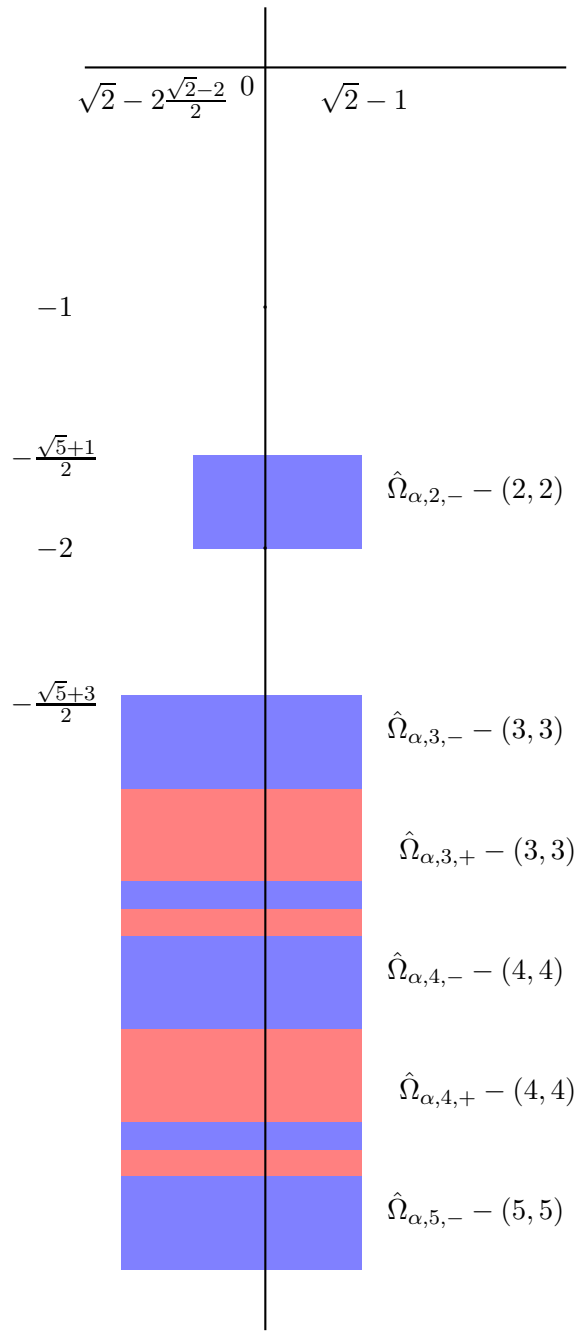
%%%%%%%%%%%%%%%%%%%%%
Although formally $\alpha\not\in {\mathbb I}_{\alpha}$, we define $k_0$ as the first digit of $\alpha$ in the $G_{\alpha}$-expansion of $\alpha$, i.e., $\frac{1}{k_{0} + \alpha}\leq\alpha < \frac{1}{k_{0} - 1 + \alpha}$. Furthermore, we define cylinders by
\begin{equation}\label{eq-3-1}
\begin{cases}
\Omega_{\alpha, k_{0}, +}^{*} & =\,\,\, \{ (x, y) \in \Omega_{\alpha}^{*} : \frac{1}{k_{0} + \alpha} < x < \alpha \};  \\[3pt]
\Omega_{\alpha, k, +}^{*}  & =\,\,\, \{ (x, y) \in \Omega_{\alpha}^{*} : \frac{1}{k + \alpha} < x \leq\frac{1}{(k-1)+\alpha} \},\,\, \text{for $k > k_{0}$};   \\[3pt]
\Omega_{\alpha, 2, -}^{*} & =\,\,\, \{ (x, y) \in \Omega_{\alpha}^{*} : \alpha - 1 < x\leq - \frac{1}{2 + \alpha} \}, \\[3pt]
\Omega_{\alpha, k, -}^{*} & =\,\,\, \{ (x, y) \in \Omega_{\alpha}^{*} : - \frac{1}{(k-1) + \alpha} < x\leq - \frac{1}{k + \alpha} \},\,\, \text{for $k \geq 3$}.
\end{cases}
\end{equation}
Then we put
\[
\left\{ \begin{array}{ccl}
\hat{\Omega}_{\alpha, k,-} & = & \left\{ (x, y) : \left(-\frac{1}{x}, - \frac{1}{y}\right) \in \Omega_{\alpha, k, -}^{*}\right\} ; \\
\hat{\Omega}_{\alpha, k,+} & = & \left\{ (x, y) : \left(\frac{1}{x},  \frac{1}{y}\right) \in \Omega_{\alpha, k, +}^{*}\right\} .
\end{array} \right.
\]
Note that if $(x,y)\in\hat{\Omega}_{\alpha,k,-}$ we have that $x>0$ and $y\geq 0$; see Figure~\ref{fig-2}.  For convenience we put $\Omega_{\alpha, k, +}^{*} = \emptyset$ for $2 \le k < k_{0}$. It is easy to see that
\[
\Omega_{\alpha}^{*} = \left( \bigcup_{k = 2}^{\infty} (\hat{\Omega}_{\alpha, k, -} - (k, k)) \right)
\cup \left( \bigcup_{k = k_{0}}^{\infty} (\hat{\Omega}_{\alpha, k, +} - (k, k))\right) \quad \text{(disj.~a.e.)},
\]
where (disj.~a.e.) means ``disjoint except for a set of measure 0''. This disjointness follows from the next lemma.
%%%%%%%%%%%%%
%%%%%%%%%%%%%
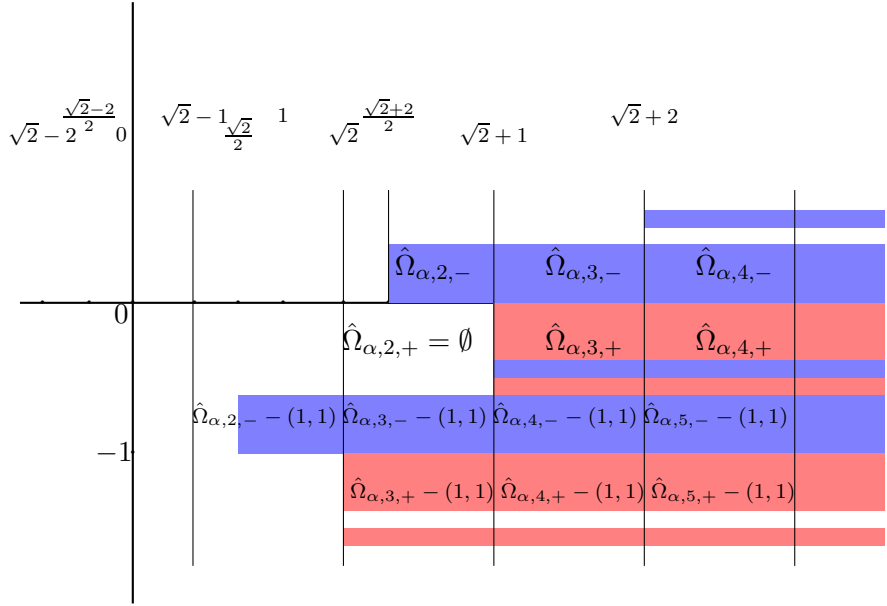
\begin{figure}
\begin{center}
\begin{tikzpicture}[scale=0.5]
\draw[thick] (-3, -4) -- (20, -4);
\node at (-2.4, 0.5){\tiny$\sqrt{2} - 2$};
\node at (-1.16, 1){\tiny$\frac{\sqrt{2} - 2}{2}$};
\node at (-0.3, 0.5){\tiny$0$}; \node at (1.64, 1){\tiny$\sqrt{2} - 1$};
\node at (2.8, 0.5){\tiny$\frac{\sqrt{2}}{ 2}$}; \node at (4, 1){\tiny$1$};
\node at (5.6, 0.5){\tiny$\sqrt{2} $};  \node at (6.8, 1){\tiny$\frac{\sqrt{2} + 2}{2}$};
\node at (9.6, 0.5){\tiny$\sqrt{2}+1 $};  \node at (13.6, 1){\tiny$\sqrt{2}+2 $};
\node at (-2.4, -4){$\cdot$};
\node at (-1.16, -4){$\cdot$};
\node at (1.64, -4){$\cdot$};
\node at (2.8, -4){$\cdot$}; \node at (4, -4){$\cdot$};
\node at (5.6, -4){$\cdot$};  \node at (6.8, -4){$\cdot$};
\node at (9.6, -4){$\cdot$};  \node at (13.6, -4){$\cdot$};

\draw[thick] (0, 4) -- (0, -12);
\fill[blue!50] (6.8, -4) -- (20, -4)--(20, -2.45)--(6.8, -2.45)--(6.8, -4);
\fill[blue!50] (13.6, -2) -- (20, -2) -- (20, -1.54) -- (13.6, -1.54) -- (13.6, -2);
\fill[red!50] (9.6, -4) -- (20, -4)--(20, -5.54)--(9.6, -5.54)--(9.6, -4);
\fill[red!50] (9.6, -6) -- (20, -6)--(20, -6.46)--(9.6, -6.46)--(9.6, -6);

\fill[blue!50] (2.8, -8) -- (20, -8)--(20, -6.45)--(2.8, -6.45)--(2.8, -8);
\fill[blue!50] (9.6, -6) -- (20, -6) -- (20, -5.54) -- (9.6, -5.54) -- (9.6, -6);
\fill[red!50] (5.6, -8) -- (20, -8)--(20, -9.54)--(5.6, -9.54)--(5.6, -8);
\fill[red!50] (5.6, -10) -- (20, -10)--(20, -10.46)--(5.6, -10.46)--(5.6, -10);
\draw (6.8, -1) -- (6.8, -4);
\draw (9.6, -1) -- (9.6, -11);
\draw (13.6, -1) -- (13.6, -11);
\draw (17.6, -1) -- (17.6, -11);
\draw (5.6, -1) -- (5.6, -11);
\draw (1.6, -1) -- (1.6, -11);
\node at (8, -3){\small$\hat{\Omega}_{\alpha, 2, -}$};
\node at (12, -3){\small$\hat{\Omega}_{\alpha, 3, -}$};
\node at (16, -3){\small$\hat{\Omega}_{\alpha, 4, -}$};

\node at (7.3, -5){\small$\hat{\Omega}_{\alpha, 2, +} = \emptyset$};
\node at (12, -5){\small$\hat{\Omega}_{\alpha, 3, +}$};
\node at (16, -5){\small$\hat{\Omega}_{\alpha, 4, +}$};

\node at (3.5, -7){\tiny$\hat{\Omega}_{\alpha, 2, -} - (1,1)$};
\node at (7.5, -7){\tiny$\hat{\Omega}_{\alpha, 3, -} - (1, 1)$};
\node at (11.5, -7){\tiny$\hat{\Omega}_{\alpha, 4, -} - (1, 1)$};
\node at (15.5, -7){\tiny$\hat{\Omega}_{\alpha, 5, -} - (1, 1)$};

\node at (7.7, -9){\tiny$\hat{\Omega}_{\alpha, 3, +} - (1, 1)$};
\node at (11.7, -9){\tiny$\hat{\Omega}_{\alpha, 4, +} - (1, 1)$};
\node at (15.7, -9){\tiny$\hat{\Omega}_{\alpha, 5, +} - (1, 1)$};

\node at (-0.5, -8) {\small$-1$};  \node at (-0.3, -4.3) {\small$0$};
\node at (0, -4){$\cdot$};\node at (0,-8){$\Huge\cdot$};
\end{tikzpicture}
\end{center}
\caption{$\hat{\Omega}_{\alpha,k,\pm} - (1,1)$ for $\alpha = \sqrt{2} - 1$}
\label{fig-3}
\end{figure}
%%%%%%%%%%%%%%%%%%%%%
%%%%%%%%%%%%%%%%%%%%%
\begin{lem} \label{lem-2} For every $k\in\N$, $k\geq 2$, we have:
\[
\left( \hat{\Omega}_{\alpha, k+1, -} - (k+1, k+1) \right) \cap
\left( \hat{\Omega}_{\alpha, k, +} - (k, k) \right) = \emptyset \quad
\text{disj.~a.e.},
\]
or equivalently,
\[
\left( \hat{\Omega}_{\alpha, k+1, -} - (1, 1) \right) \cap \left( \hat{\Omega}_{\alpha, k, +} - (0, 0) \right) = \emptyset \quad
\text{disj.~a.e.};
\]
see Figure~\ref{fig-3}.
\end{lem}

\begin{proof}
We see
\[
\left( \hat{\Omega}_{\alpha, k, \pm} - (k, k) \right) = \hat{G}^{*}_{\alpha}\left(\Omega_{\alpha, k,\pm}^{*} \right) .
\]
Then the assertion follows from the a.e.-bijectivity of $\hat{G}_{\alpha}^{*}$.
\end{proof}

For $j \ge 1$, we define 
\[
\Upsilon_{\alpha, j} = \bigcup_{k=j+1}^{\infty}\left(\hat{\Omega}_{\alpha, k, -} - (k-j, k-j)\right)
\cup
\bigcup_{k=j+1}^{\infty}\left(\hat{\Omega}_{\alpha, k, +} - (k-j, k-j)\right)
\]    
for $j \ge 2$, see Figure~\ref{fig-4}, and
\[
\Upsilon_{\alpha} = \bigcup_{j=1}^{\infty} \Upsilon_{\alpha, j}  .
\]
From Lemma~\ref{lem-2}, this is ``disj.~a.e.".  We also see
\[
\Upsilon_{\alpha, j}\cap \hat{\Omega}_{\alpha, j, +} = \emptyset
\qquad \text{(disj.~a.e.)} ,
\]
which implies
\[
\Omega^{*}_{\alpha} \cap \left( \Upsilon_{\alpha} \right)^{-1} = \emptyset \qquad \text{(disj.~a.e.)},
\]
where $\left( \Upsilon_{\alpha} \right)^{-1} = \left\{ (x,y)\, :\, \left( \tfrac{1}{x},\tfrac{1}{y}\right) \in \Upsilon_{\alpha}\right\}$. Note that $(\Upsilon_{\alpha})^{-1} \subset \{(x, y) : x > 0\}$.

Now we will define the ‘reasonable domain' $V_{\alpha}$ for the natural extension map $\hat{F}_{\alpha}$ from~\eqref{eq:9}.  We put $V_{\alpha} = \Omega_{\alpha}^{*}\cup \left(\Upsilon_{\alpha}\right)^{-1}$; see Figure~\ref{fig-5}. From the construction of $V_{\alpha}$, it is not hard to see the following result.
%%%%%%%%%%%%%
%%%%%%%%%%%%%
\begin{figure}
\begin{center}
\begin{tikzpicture}[scale=0.5]
\draw[thick] (-3, -4) -- (20, -4);
\node at (-2.4, -3.5){\tiny$\sqrt{2} - 2$};
\node at (-1.16, -3){\tiny$\frac{\sqrt{2} - 2}{2}$};
\node at (1.64, -3){\tiny$\sqrt{2} - 1$};
\node at (2.8, -3.5){\tiny$\frac{\sqrt{2}}{ 2}$}; \node at (4, -3){\tiny$1$};
\node at (5.6, -3.5){\tiny$\sqrt{2} $};  \node at (6.8, -3){\tiny$\frac{\sqrt{2} + 2}{2}$};
\node at (9.6, -3.5){\tiny$\sqrt{2}+1 $};  \node at (13.6, -3){\tiny$\sqrt{2}+2 $};
\node at (-2.4, -4){$\cdot$};
\node at (-1.16, -4){$\cdot$};
\node at (1.64, -4){$\cdot$};
\node at (2.8, -4){$\cdot$}; \node at (4, -4){$\cdot$};
\node at (5.6, -4){$\cdot$};  \node at (6.8, -4){$\cdot$};
\node at (9.6, -4){$\cdot$};  \node at (13.6, -4){$\cdot$};

\draw[thick] (0, -2) -- (0, -21);

\fill[blue!50] (2.8, -8) -- (20, -8)--(20, -6.45)--(2.8, -6.45)--(2.8, -8);
\fill[blue!50] (9.6, -6) -- (20, -6) -- (20, -5.54) -- (9.6, -5.54) -- (9.6, -6);
\fill[red!50] (5.6, -8) -- (20, -8)--(20, -9.54)--(5.6, -9.54)--(5.6, -8);
\fill[red!50] (5.6, -10) -- (20, -10)--(20, -10.46)--(5.6, -10.46)--(5.6, -10);

\fill[blue!50] (1.6, -12) -- (20, -12)--(20, -10.45)--(1.6, -10.45)--(1.6, -12);
\fill[blue!50] (5.6, -10) -- (20, -10) -- (20, -9.54) -- (5.6, -9.54) -- (5.6, -10);
\fill[red!50] (1.6, -12) -- (20, -12)--(20, -13.54)--(1.6, -13.54)--(1.6, -12);
\fill[red!50] (1.6, -14) -- (20, -14)--(20, -14.46)--(1.6, -14.46)--(1.6, -14);

\fill[blue!50] (1.6, -16) -- (20, -16)--(20, -14.45)--(1.6, -14.45)--(1.6, -16);
\fill[blue!50] (1.6, -14) -- (20, -14) -- (20, -13.54) -- (1.6, -13.54) -- (1.6, -14);
\fill[red!50] (1.6, -16) -- (20, -16)--(20, -17.54)--(1.6, -17.54)--(1.6, -16);
\fill[red!50] (1.6, -18) -- (20, -18)--(20, -18.46)--(1.6, -18.46)--(1.6, -18);

\fill[blue!50] (1.6, -20) -- (20, -20)--(20, -18.45)--(1.6, -18.45)--(1.6, -20);
\fill[blue!50] (1.6, -18) -- (20, -18) -- (20, -17.54) -- (1.6, -17.54) -- (1.6, -18);

\draw (9.6, -4) -- (9.6, -21);
\draw (13.6, -4) -- (13.6, -21);
\draw (17.6, -4) -- (17.6, -21);
\draw (5.6, -4) -- (5.6, -21);
\draw (1.6, -4) -- (1.6, -21);

\node at (3.5, -7){\tiny$\hat{\Omega}_{\alpha, 2, -} - (1,1)$};
\node at (7.5, -7){\tiny$\hat{\Omega}_{\alpha, 3, -} - (1, 1)$};
\node at (11.5, -7){\tiny$\hat{\Omega}_{\alpha, 4, -} - (1, 1)$};
\node at (15.5, -7){\tiny$\hat{\Omega}_{\alpha, 5, -} - (1, 1)$};
\node at (7.7, -9){\tiny$\hat{\Omega}_{\alpha, 3, +} - (1, 1)$};
\node at (11.7, -9){\tiny$\hat{\Omega}_{\alpha, 4, +} - (1, 1)$};
\node at (15.7, -9){\tiny$\hat{\Omega}_{\alpha, 5, +} - (1, 1)$};

\node at (3.5, -11){\tiny$\hat{\Omega}_{\alpha, 3, -} - (2, 2)$};
\node at (7.5, -11){\tiny$\hat{\Omega}_{\alpha, 4, -} - (2, 2)$};
\node at (11.5, -11){\tiny$\hat{\Omega}_{\alpha, 5, -} - (2, 2)$};
\node at (15.5, -11){\tiny$\hat{\Omega}_{\alpha, 6, -} - (2, 2)$};

\node at (3.7, -13){\tiny$\hat{\Omega}_{\alpha, 3, +} - (2, 2)$};
\node at (7.7, -13){\tiny$\hat{\Omega}_{\alpha, 4, +} - (2, 2)$};
\node at (11.7, -13){\tiny$\hat{\Omega}_{\alpha, 5, +} - (2, 2)$};
\node at (15.7, -13){\tiny$\hat{\Omega}_{\alpha, 6, +} - (2, 2)$};

\node at (3.5, -15){\tiny$\hat{\Omega}_{\alpha, 4, -} - (3, 3)$};
\node at (7.5, -15){\tiny$\hat{\Omega}_{\alpha, 5, -} - (3, 3)$};
\node at (11.5, -15){\tiny$\hat{\Omega}_{\alpha, 6, -} - (3, 3)$};
\node at (15.5, -15){\tiny$\hat{\Omega}_{\alpha, 7, -} - (3, 3)$};

\node at (3.7, -17){\tiny$\hat{\Omega}_{\alpha, 4, +} - (3, 3)$};
\node at (7.7, -17){\tiny$\hat{\Omega}_{\alpha, 5, +} - (3, 3)$};
\node at (11.7, -17){\tiny$\hat{\Omega}_{\alpha, 6, +} - (3, 3)$};
\node at (15.7, -17){\tiny$\hat{\Omega}_{\alpha, 7, +} - (3, 3)$};

\node at (3.5, -19){\tiny$\hat{\Omega}_{\alpha, 5, -} - (4, 4)$};
\node at (7.5, -19){\tiny$\hat{\Omega}_{\alpha, 6, -} - (4, 4)$};
\node at (11.5, -19){\tiny$\hat{\Omega}_{\alpha, 7, -} - (4, 4)$};
\node at (15.5, -19){\tiny$\hat{\Omega}_{\alpha, 8, -} - (4, 4)$};

\node at (3.5, -22){$\Upsilon{\alpha, 1}$};
\node at (7.5, -22){$\Upsilon{\alpha, 2}$};
\node at (11.5, -22){$\Upsilon{\alpha, 3}$};
\node at (15.5, -22){$\Upsilon{\alpha, 4}$};

\node at (-0.5, -8) {\small$-1$};  \node at (-0.3, -4.3) {\small$0$};
\node at (0, -4){$\cdot$};\node at (0,-8){$\Huge\cdot$};
\end{tikzpicture}
\end{center}
\caption{$\hat{\Omega}_{\alpha,k,\pm} - (\ell,\ell)$ and $\Upsilon_{\alpha, k}$ for $\alpha = \sqrt{2} - 1$}
\label{fig-4}
\end{figure}
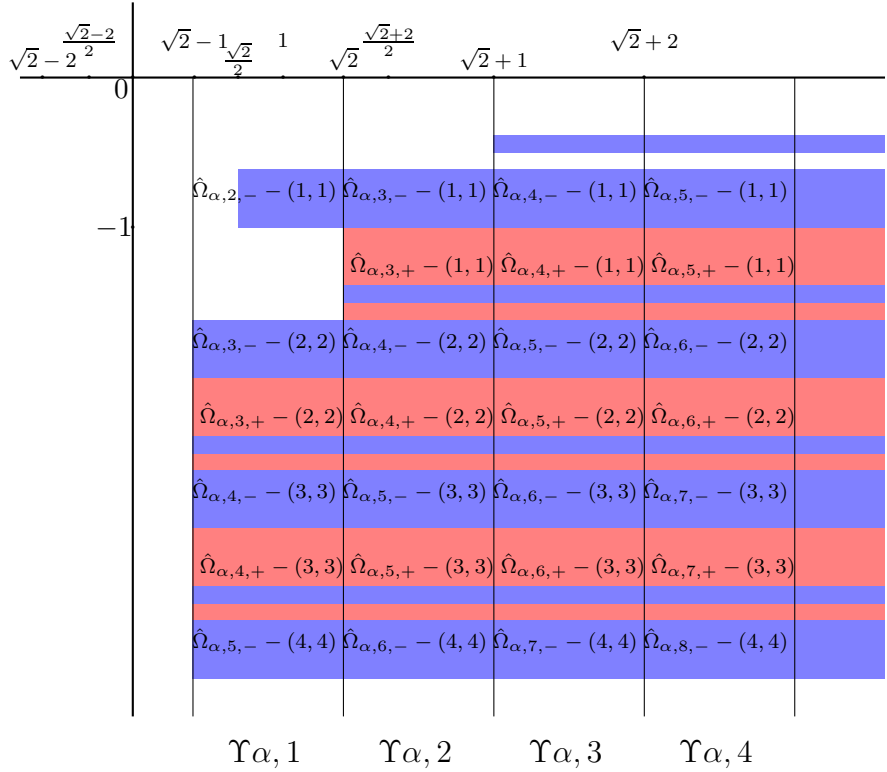
%%%%%%%%%%%%%%%%%%%%%
%%%%%%%%%%%%%%%%%%%%%
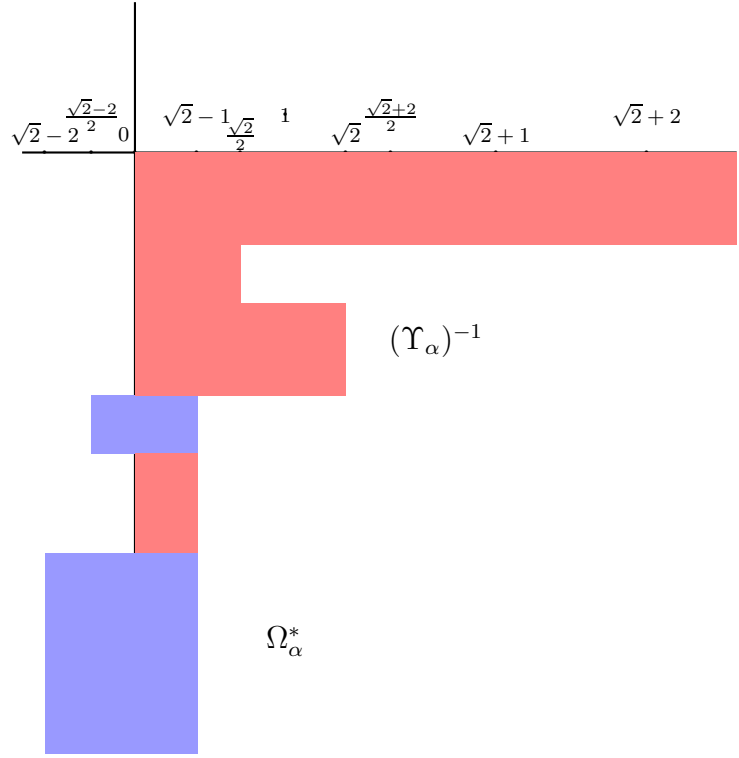
\begin{figure}
\begin{center}
\begin{tikzpicture}[scale=0.5]
\draw[thick] (-3, 0) -- (16, 0);
\node at (-2.4, 0.5){\tiny$\sqrt{2} - 2$};
\node at (-1.16, 1){\tiny$\frac{\sqrt{2} - 2}{2}$};
\node at (-0.3, 0.5){\tiny$0$}; \node at (1.64, 1){\tiny$\sqrt{2} - 1$};
\node at (2.8, 0.5){\tiny$\frac{\sqrt{2}}{ 2}$}; \node at (4, 1){\tiny$1$};
\node at (5.6, 0.5){\tiny$\sqrt{2} $};  \node at (6.8, 1){\tiny$\frac{\sqrt{2} + 2}{2}$};
\node at (9.6, 0.5){\tiny$\sqrt{2}+1 $};  \node at (13.6, 1){\tiny$\sqrt{2}+2 $};
\node at (-2.4, 0){$\cdot$};
\node at (-1.16, 0){$\cdot$};
\node at (0,0){$\cdot$}; \node at (1.64, 0){$\cdot$};
\node at (2.8, 0){$\cdot$}; \node at (4, 1){$\cdot$};
\node at (5.6, 0){$\cdot$};  \node at (6.8, 0){$\cdot$};
\node at (9.6, 0){$\cdot$};  \node at (13.6, 0){$\cdot$};
\draw[thick] (0, 4) -- (0, -16);
\fill[blue!40] (-2.4, -16) -- (1.66, -16)--(1.66, -10.66)--(-2.4, -10.66)--(-2.4, -16);
\fill[blue!40] (-1.16, -8) -- (1.66, -8) -- (1.66, -6.45) -- (-1.16, -6.45) -- (-1.16, -8);
\fill[red!50] (0, 0) -- (16,0)--(16, -2.45)--(0, -2.45)--(0, 0);
\fill[red!50](2.8, -4)--(0, -4)--(0, -2.45)--(2.8, -2.45)--(2.8, -4);
\fill[red!50] (5.6, -4) -- (0, -4)--(0, -6.46)--(5.6, -6.46)--(5.6, -4);
\fill[red!50] (0, -8) -- (1.66, -8) -- (1.66, -10.66) -- (0, -10.66) -- (0, -8);

\node at (4, -13){$\Omega_{\alpha}^{*}$};
\node at (8, -5){($\Upsilon_{\alpha})^{-1}$};
\end{tikzpicture}
\end{center}
\caption{$V_{\alpha}= \Omega_{\alpha}^{*}\cup (\Upsilon_{\alpha})^{-1}$ for $\alpha = \sqrt{2} - 1$}
\label{fig-5}
\end{figure}
%%%%%%%%%%%%%%%%%%%%%
%%%%%%%%%%%%%%%%%%%%%

\begin{thm}\label{thm-1}
The dynamical system $(V_{\alpha}, \hat{F}_{\alpha})$ together with the measure $\mu_{\alpha}$ with density $\tfrac{dx dy}{(x - y)^{2}}$ is a representation of the natural extension of $(\big[ \alpha - 1, \tfrac{1}{\alpha}\big) , F_{\alpha})$ with measure $\nu_{\alpha}$, which is the projection of $\mu_{\alpha}$ on the first coordinate.
\end{thm}

\begin{proof}
We show the following below.  Then the assertion of the theorem is proved in exactly the same way as in~\cite{Nats-2} in the case of $\tfrac{1}{2} \leq\alpha\leq 1$.
\begin{enumerate}
\item
The map $\hat{F}_{\alpha}$ defined on $V_{\alpha}$ is surjective.
\item
The map $\hat{F}_{\alpha}$ is bijective except for a set of measure $0$.
\item
The measure $\frac{dx dy}{(x - y)^{2}}$ is the absolutely continuous ergodic invariant measure.
\item
The Borel $\sigma$-algebra ${\mathcal B}(V_{\alpha})$ on $V_{\alpha}$ satisfies:
$$
{\mathcal B}(V_{\alpha}) = \sigma \left( \bigvee_{n=0}^{\infty} \hat{F}_{\alpha}^n\pi_1^{-1}{\mathcal B}([\alpha -1,\alpha))\right),
$$
where ${\mathcal B}([\alpha -1,\alpha))$ is the Borel $\sigma$-algebra on $[\alpha -1,\alpha )$ and $\pi_1: V_{\alpha}\to [\alpha -1,\alpha)$ is the projection on the first coordinate.
\item[]
\end{enumerate}

For a.e.~$(x, y)\in\Omega_{\alpha}^{*}$, $x\neq 0$,  there exists a unique element $(x_{0}, y_{0}) \in\Omega_{\alpha}^{*}$ and a positive integer $k$ such that
\[
(x, y) = \hat{G}_{\alpha}^{*}(x_{0}, y_{0}) = \begin{cases}
\left(-\frac{1}{x_{0}} - k, -\frac{1}{y_{0}} - k\right), & \text{if $x_{0} < 0$}; \\
\left(\frac{1}{x_{0}} - k, \frac{1}{y_{0}} - k\right), & \text{if $x_{0} > 0$},
\end{cases}
\]
since $(\Omega_{\alpha}^{*}, \hat{G}_{\alpha}^{*})$ is a natural extension of $(\mathbb I_{\alpha}, G_{\alpha})$; see~\cite{K-S-S}.

If $x_{0} < 0$, then we see
\[
\left(-\frac{1}{x_{0}} - (k-1), -\frac{1}{y_{0}} - (k-1)\right) \in \Upsilon_{\alpha, 1} .
\]
This implies $\alpha \le -\frac{1}{x_{0}} - (k-1) < \alpha + 1$. We put
$$
(x_{1}, y_{1}) = \left(\left(-\frac{1}{x_{0}} - (k-1)\right)^{-1}, \left(-\frac{1}{y_{0}} - (k-1)\right)^{-1} \right),
$$
and have $\frac{1}{1 + \alpha} < x_{1} \le \frac{1}{\alpha}$. From~\eqref{eq:9}, we have that $(x, y) = \hat{F}_{\alpha}(x_{1}, y_{1})$.

If $x_{0} > 0$, then
\[
\left(\frac{1}{x_{0}} - (k-1), \frac{1}{y_{0}} - (k-1)\right)\in \Upsilon_{\alpha, 1}
\]
and $(x, y) = \hat{F}_{\alpha}(x_{1}, y_{1})$ with 
$$
(x_{1}, y_{1}) = \left(\left(\frac{1}{x_{0}} - (k-1)\right)^{-1}, \left(\frac{1}{y_{0}} - (k-1)\right)^{-1} \right).
$$  
We note that in both cases we have
\begin{equation} \label{eq-3-3}
(x_{1}, y_{1}) \in \Upsilon_{\alpha, 1}^{-1} \subset \left(\Upsilon_{\alpha}\right)^{-1}.
\end{equation}

Next we consider the case $(x, y) \in \Upsilon_{\alpha}^{-1}$. This means $\left( \tfrac{1}{x}, \tfrac{1}{y}\right) \in \Upsilon_{\alpha, j}$ for some $j\geq 1$.  We consider two cases.\medskip\

Case (a): $\left(\tfrac{1}{x}, \tfrac{1}{y}\right) \in \hat{\Omega}_{\alpha, j+1, \pm} - (1, 1)$.\\
In this case, there exists $(x_{0}, y_{0}) \in \hat{\Omega}_{\alpha, j+1, \pm}$ such
that $\left( \tfrac{1}{x}, \tfrac{1}{y}\right) = (x_{0} - 1, y_{0} - 1)$.  Thus
$(x_{1}, y_{1}) = \left(\pm \tfrac{1}{x_{0}}, \pm \tfrac{1}{y_{0}}\right) \in
\Omega_{\alpha, j+1}^{*}$, which imlies
\[
\frac{1}{j + 1 + \alpha} < \pm x_{1} \le \frac{1}{j + \alpha} \le \frac{1}{1 + \alpha} .
\]
Hence we find that  $(x, y) = \hat F_{\alpha}(x_{1}, y_{1})$.  Here we see
\begin{equation} \label{eq-3-4}
(x_{1}, y_{1}) \in \Omega^{*}_{\alpha} .
\end{equation}

Case (b): $\left(\tfrac{1}{x}, \tfrac{1}{y}\right) \in \hat{\Omega}_{\alpha, k, \pm} -
(k - j, k - j)$ for $k > j+1$.\\
In this case, we have
\begin{equation} \label{eq-3-5}
(x_{0}, y_{0}) := \left(\tfrac{1}{x} + 1, \tfrac{1}{y} + 1\right)
\in \Omega_{\alpha, k \pm} - (k - j - 1, k - j - 1)
\end{equation}
and then
\begin{equation} \label{eq-3-6}
(x_{1}, y_{1}) := (\tfrac{1}{x_{0}}, \tfrac{1}{y_{0}}) \in \Upsilon_{\alpha}^{-1} .
\end{equation}
This shows
\[
(x, y) = \hat{F}_{\alpha}(x_{1}, y_{1}) .
\]
Consequently, we have the first statement.  The second statement follows from~\eqref{eq-3-3}, \eqref{eq-3-4}, \eqref{eq-3-5} and~\eqref{eq-3-6}.
The third statement is also easy to obtain.  Indeed, it is well-known that the measure given here is the absolutely continuous invariant measure for the direct product of the same linear fractional transformation.  Because $\hat{G}_{\alpha}^{*}$ is an induced transformation of $\hat{F}_{\alpha}$ to $\Omega_{\alpha}^{*}$, the ergodicity of $\hat{F}_{\alpha}$ follows from that of $\hat{G}_{\alpha}^{*}$.  The ergodicity of the latter is equivalent to that of $G_{\alpha}$ and it was proved by  L.~Luzzi and S.~Marmi in~\cite{L-M}.  For the last statement, note that~\eqref{eq-add-1} allows one to identify a point $x$ with a one-sided sequence of matrices with entries $A^{\pm}$, $A^R$. As a result, $F_{\alpha}$, $\hat{F}_{\alpha}$ can be seen as a one-sided resp.\ two sided shifts, from which the fourth statement follows.
\end{proof}

%%%%%%%%%%%%%%%%
%%%%%%%%%%%%%%%%
\section{The natural extension of $F_{\alpha, \flat}$ for $0 < \alpha < \frac{1}{2}$}
In the case of $\alpha = 1$, $F_{1}$ is the original Farey map $F$. We recall that
\[
\hat{F}(x, y) =\begin{cases}
\left(\frac{x}{1-x}, \frac{y}{1 - y}\right), & \text{if $0\leq x < \frac{1}{2}$}; \\
\left(\frac{1 - x}{x}, \frac{1 - y}{ y}\right), & \text{if $\frac{1}{2}\leq x \leq 1$},
\end{cases}
\]
defined on $V_{1}= \{ (x, y) : 0 \le x \le 1, \,\, -\infty \le y \le  0 \}$ is
the natural extension of $F$ with the invariant measure $\hat{\mu}_{1}$ defined by $d\hat{\mu_{1}} = \frac{dx \, dy}{(x -y)^{2}}$. In particular, $\hat{F}$ is bijective on $V_{1}$ except for a set of Lebesgue measure $0$.  It is easy to see that $F_{1}$ and $F_{1, \flat}$ are the same. In the case of $\frac{1}{2} \le \alpha < 1$, the complete description was given in~\cite{Nats-2}. Here we consider the case $0 < \alpha < \tfrac{1}{2}$ as a continuation of the previous section.

We put $V_{\alpha, \flat} = \{(x, y) \in V_{\alpha}, \,\, x \le 1\}$ and consider the induced transformation $\hat{F}_{\alpha, \flat}$ of $\hat{F}_{\alpha}$ to $V_{\alpha, \flat}$; see Figure~\ref{fig-6}.
%%%%%%%%%%%%%%%%%%%%%
%%%%%%%%%%%%%%%%%%%%%
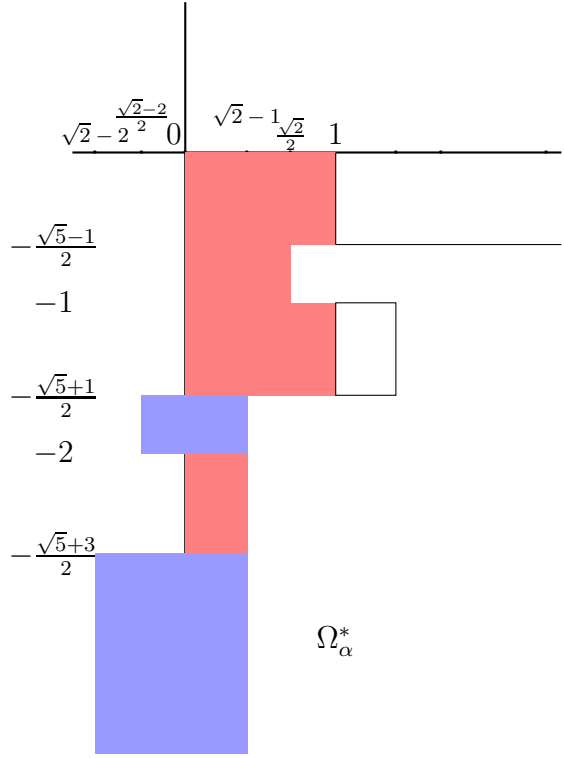
\begin{figure}
\begin{center}
\begin{tikzpicture}[scale=0.5]
\draw[thick] (-3, 0) -- (10, 0);
\node at (-2.4, 0.5){\tiny $\sqrt{2} - 2$};
\node at (-1.16, 1){\tiny$\frac{\sqrt{2} - 2}{2}$};
\node at (-0.3, 0.5){$0$}; \node at (1.64, 1){\tiny$\sqrt{2} - 1$};
\node at (2.8, 0.5){\tiny$\frac{\sqrt{2}}{ 2}$}; \node at (4, 0.5){$1$};

\node at (-3.5, -8)  {$-2$}; \node at (-3.5, -10.66) {$- \frac{\sqrt{5} + 3}{2}$};
\node at (-3.5, -6.45) {$-\frac{\sqrt{5}+1}{2}$};
\node at (-3.5, -4) {$-1$};
\node at (-3.5, -2.45) {$-\frac{\sqrt{5} - 1}{2}$};

\node at (-2.4, 0){$\cdot$};
\node at (-1.16, 0){$\cdot$};
\node at (0,0){$\cdot$}; \node at (1.64, 0){$\cdot$};
\node at (2.8, 0){$\cdot$};
\node at (5.6, 0){$\cdot$};  \node at (6.8, 0){$\cdot$};
\node at (9.6, 0){$\cdot$};  \node at (13.6, 0){$\cdot$};
\draw[thick] (0, 4) -- (0, -16);
\fill[blue!40] (-2.4, -16) -- (1.66, -16)--(1.66, -10.66)--(-2.4, -10.66)--(-2.4, -16);
\fill[blue!40] (-1.16, -8) -- (1.66, -8) -- (1.66, -6.45) -- (-1.16, -6.45) -- (-1.16, -8);
\fill[red!50] (0, 0) -- (4,0)--(4, -2.45)--(0, -2.45)--(0, 0);
\draw (10, -2.45)--(4, -2.45);
\draw (4, 0) -- (4, -2.45);
\fill[red!50](2.8, -4)--(0, -4)--(0, -2.45)--(2.8, -2.45)--(2.8, -4);
\fill[red!50] (4, -4) -- (0, -4)--(0, -6.46)--(4, -6.46)--(4, -4);
\draw (5.6, -4) -- (4, -4)--(4, -6.46)--(5.6, -6.46)--(5.6, -4) --cycle;

\fill[red!50] (0, -8) -- (1.66, -8) -- (1.66, -10.66) -- (0, -10.66) -- (0, -8);

\node at (4, -13){$\Omega_{\alpha}^{*}$};
%\node at (8, -5){($\Upsilon_{\alpha})^{-1}$};
\end{tikzpicture}
\end{center}
\caption{$V_{\alpha, \flat}= V_{\alpha} \cap \{(x, y) : x \le 1 \}$ for $\alpha = \sqrt{2}-1$}
\label{fig-6}
\end{figure}
%%%%%%%%%%%%%%%%%%%%%
%%%%%%%%%%%%%%%%%%%%%
Recall the definition of the map $\hat{F}_{\alpha,\flat}$, as given in~\eqref{eq:10}.
%%%%%%%%%%%%
%%%%%%%%%%%%
\begin{thm} \label{thm-2}
The dynamical system $(V_{\alpha,\flat},\hat{F}_{\alpha,\flat},\mu_{\alpha,\flat})$ is a representation of the natural extension of $([\alpha -1,1], F_{\alpha, \flat}, \nu_{\alpha, \flat})$. Here the invariant measure $\mu_{\alpha, \flat}$ has density $\frac{dx dy}{(x-y)^{2}}$ on $V_{\alpha, \flat}$, and $\nu_{\alpha, \flat}$ is the projection of $\mu_{\alpha, \flat}$ on the first coordinate.
\end{thm}

\begin{proof}
Recall that $F_{\alpha, \flat}(x) = F_{\alpha}^{K(x)}(x)$ with $K(x) = \min\{ K \ge 1 : F_{\alpha}^{K}(x) \in [\alpha - 1, 1]\}$, and for $(x, y) \in V_{\alpha, \flat}$, one has $x \in [\alpha - 1, 1]$. Since the first coordinate of $\hat{F}_{\alpha}(x, y)$ is $F_{\alpha}(x)$, we find
$\hat{F}_{\alpha, \flat}(x, y) = \hat{F}_{\alpha}^{K(x)}(x, y)$. Here, we note that the first coordinate is $F_{\alpha}(x)$. Because of the general fact that an induced transformation of a bijective map is bijective, we see that $\hat{F}_{\alpha, \flat}$ is bijective.  The rest of the proof follows from a standard argument.
\end{proof}

Put
\[
\left\{ \begin{array}{lcll}
D_{1} & = & \left(V_{\alpha,-} + (1, 1)\right) & \left(\subset V_{1}\right) \\
D_{2} & = & V_{\alpha, +} \setminus D_{1} & \left(\subset V_{1}\right)
\end{array} \right.
\]
with $V_{\alpha, -} = \{(x, y) \in V_{\alpha, \flat} : x <0\}$ and $V_{\alpha, +} = \{(x, y) \in V_{\alpha, \flat} : x \ge 0\}$.   We write
$D = D_{1} \cup D_{2}$.  By the definition we see that $D \subset V_{1}$. We define $\psi\, : \, D \to V_{\alpha}$ by
\begin{equation}\label{PSI}
\psi (x, y) = \begin{cases}
(x -1, y - 1), & \text{if $(x, y) \in D_{1}$}; \\
(x, y), & \text{if $(x, y) \in D_{2}$};
\end{cases}
\end{equation}
see Figure~\ref{fig-7}. We note that
\begin{enumerate}
\item
$D$ has positive Lebesgue measure since $V_{\alpha, -}$ has\\ 
positive Lebesgue measure; see~\cite{K-S-S};
\item
$(\psi)^{-1} (\mu_{\alpha}) = \mu_{1}|_{D}$;
\item
$\psi$ is injective.
\end{enumerate}
%%%%%%%%%%%%%%%%%%%%%%%%
%%%%%%%%%%%%%%%%%%%%%
%%%%%%%%%%%%%%%%%%%%%
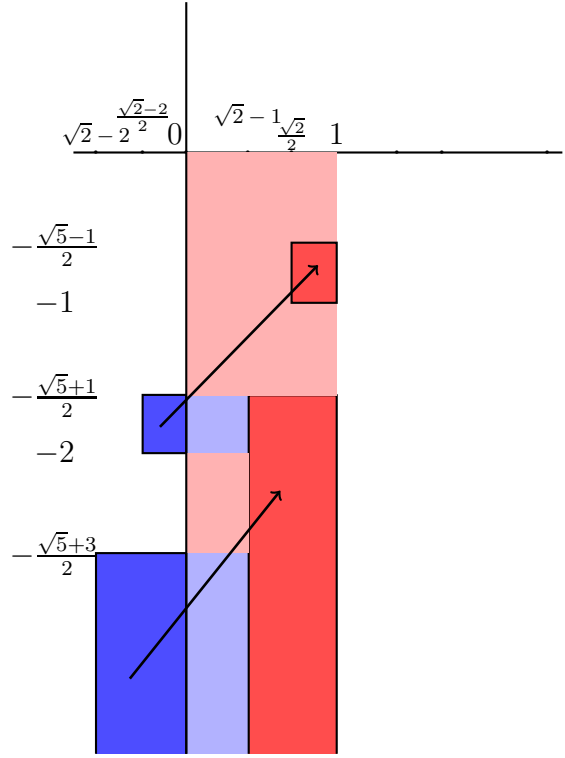
\begin{figure}
\begin{center}
\begin{tikzpicture}[scale=0.5]
\draw[thick] (-3, 0) -- (10, 0);
\node at (-2.4, 0.5){\tiny $\sqrt{2} - 2$};
\node at (-1.16, 1){\tiny$\frac{\sqrt{2} - 2}{2}$};
\node at (-0.3, 0.5){$0$}; \node at (1.64, 1){\tiny$\sqrt{2} - 1$};
\node at (2.8, 0.5){\tiny$\frac{\sqrt{2}}{ 2}$}; \node at (4, 0.5){$1$};

\node at (-3.5, -8)  {$-2$}; \node at (-3.5, -10.66) {$- \frac{\sqrt{5} + 3}{2}$};
\node at (-3.5, -6.45) {$-\frac{\sqrt{5}+1}{2}$};
\node at (-3.5, -4) {$-1$};
\node at (-3.5, -2.45) {$-\frac{\sqrt{5} - 1}{2}$};

\node at (-2.4, 0){$\cdot$};
\node at (-1.16, 0){$\cdot$};
\node at (0,0){$\cdot$}; \node at (1.64, 0){$\cdot$};
\node at (2.8, 0){$\cdot$};
\node at (5.6, 0){$\cdot$};  \node at (6.8, 0){$\cdot$};
\node at (9.6, 0){$\cdot$};  \node at (13.6, 0){$\cdot$};

\fill[blue!30] (0, -16) -- (1.66, -16)--(1.66, -10.66)--(0, -10.66)--(0, -16);
\fill[blue!70] (-2.4, -16) -- (0, -16)--(0, -10.66)--(-2.4, -10.66)--(-2.4, -16);
\draw[thick]  (0, -16)--(0, -10.66)--(-2.4, -10.66)--(-2.4, -16);

\fill[blue!70] (-1.16, -8) -- (0, -8) -- (0, -6.45) -- (-1.16, -6.45) -- (-1.16, -8);
\fill[blue!30] (0, -8) -- (1.66, -8) -- (1.66, -6.45) -- (0, -6.45) -- (0, -8);
\draw[thick] (0, -8) -- (-1.16, -8) -- (-1.16, -6.45) -- (0, -6.45) -- (0, -8)--cycle;

\fill[red!30] (0, 0) -- (4,0)--(4, -2.45)--(0, -2.45)--(0, 0);

\fill[red!70] (1.66, -6.46) -- (4, -6.46)--(4, -16)--(1.66, -16)--(1.66, -6.46);
\draw[thick] (1.66, -16)--(1.66, -6.46) -- (1.66, -6.46) -- (4, -6.46)--(4, -16);

\fill[red!30](2.8, -4)--(0, -4)--(0, -2.45)--(2.8, -2.45)--(2.8, -4);
\fill[red!30] (4, -4) -- (0, -4)--(0, -6.46)--(4, -6.46)--(4, -4);

\fill[red!70] (2.8, -4) -- (4, -4)--(4, -2.4)--(2.8, -2.4)--(2.8, -4);
\draw[thick] (2.8, -4) -- (4, -4)--(4, -2.4)--(2.8, -2.4)--(2.8, -4) --cycle;

\fill[red!30] (0, -8) -- (1.66, -8) -- (1.66, -10.66) -- (0, -10.66) -- (0, -8);

\draw[thick] (0, 4) -- (0, -16);
\draw[line width=1pt,  ->] (-0.7, -7.3) -- (3.5, -3);
\draw[line width=1pt,  ->] (-1.5, -14) -- (2.5, - 9);

\end{tikzpicture}
\end{center}
\caption{$\psi^{-1}$ for $\alpha = \sqrt{2} -1$}
\label{fig-7}
\end{figure}
%%%%%%%%%%%%%%%%%%%%%
%%%%%%%%%%%%%%%%%%%%%
\begin{thm}\label{thm-3}
We have $D = V_{1}$ and for a.e. $(x, y) \in V_{1}$,
\begin{equation} \label{eq-4-2}
\left((\psi)^{-1} \circ \hat{F}_{\alpha, \flat}\circ \psi\right) (x, y)= \hat{F}(x, y) .
\end{equation}
In other words, for any $0 < \alpha < \frac{1}{2}$, $\left( V_{\alpha, \flat},
\hat{\mu}_{\alpha, \flat},
 \hat{F}_{\alpha, \flat} \right)$ is metrically isomorphic to $\left(V_{1},
\hat{\mu}_{1},  \hat{F} \right)$ by the isomorphism $\psi : V_{1} \to
V_{\alpha, \flat}$.
\end{thm}

\begin{proof}
We choose $(x, y) \in D$ such that both $\{\hat{F}^{n}(x, y) : n \ge 0\} $ and 
$\{\hat{F}_{\alpha, \flat}^{n}(\psi(x, y)) : n \ge 0\}$
are dense in $V_{1}$ and $V_{\alpha}$, respectively.  This is possible due to the fact that $\hat{F}$ and
$\hat{F}_{\alpha, \flat}$ are ergodic with respect to $\mu_{1}$ and $\mu_{\alpha}$,
respectively. We see that the following hold:
\begin{enumerate}
\item
$(x_{0}, y_{0}) = \psi (x, y) \in V_{\alpha, -}$,  $\alpha \le x < \frac{1}{2}$\\
In this case, $F_{\alpha, \flat}(x_{0}) = \frac{2x -1}{1 -x} < 0$ since $x_{0} < - \frac{1}{2}$.  Then we see $(\psi^{-1} \circ \hat{F}_{\alpha, \flat}\circ \psi)(x, y) =
(\frac{x}{1 -x}, \frac{y}{1 - y}) = \hat{F}(x, y)$.

\item
$(x_{0}, y_{0}) = \psi(x, y) \in V_{\alpha, -}$,  $\frac{1}{2} \le x < 1$\\
For $F_{\alpha, \flat}(x_{0}) = \frac{1 - x}{x} > 0$, we see
$(\psi^{-1} \circ \hat{F}_{\alpha, \flat}\circ \psi)(x, y) = (\frac{1 -x}{x},
\frac{1 - y}{y}) = \hat{F}(x, y)$.

\item
$(x, y) \notin  \psi^{-1}(V_{\alpha, -})$,  $0\le  x \le \frac{1}{2}$\\
In this case, $\hat{F}_{\alpha, \flat}(\psi(x, y)) = \hat{F}_{\alpha, \flat}(x, y) =
\left( \frac{x}{1 -x},
\frac{y}{1 - y} \right)$ and $\frac{x }{1 -x} \ge 0$.  Thus we have
$(\psi^{-1}\circ \hat{F}_{\alpha, \flat}\circ \psi)(x, y) =\hat{F}(x, y)$.

\item
$(x, y) \notin  \psi^{-1}(V_{\alpha, -})$,  $\frac{1}{2} \le  x \le \frac{1}{1+ \alpha}$\\
We see $\psi(x, y) = (x, y)$ again and $\hat{F}_{\alpha, \flat}(\psi(x, y)) =
\hat{F}_{\alpha, \flat}(x, y) = \left(\frac{1 -2x}{x},
\frac{1 - 2y}{y} \right)$.  However, $\frac{1 -2x}{x} < 0$.  So we have
\begin{eqnarray*}
\left(\psi^{-1} \circ \hat{F}_{\alpha, \flat} \circ \psi\right) (x, y) &=&
\left(\frac{1 -2x}{x} + 1, \frac{1 - 2y}{y} + 1 \right) \\
&=&
\left(\frac{1 -x}{x}, \frac{1 - y}{y} \right) \, =\, \hat{F}(x,y).
\end{eqnarray*}
\item
$(x, y) \notin  \psi^{-1}(V_{\alpha, -})$,  $\frac{1}{1+ \alpha}< x \le 1$ \\
In this case, we see
\[
\hat{F}_{\alpha, \flat}(\psi(x, y)) = \hat{F}_{\alpha, \flat}(x, y) =
\left(\frac{1 -x}{x}, \frac{1 - y}{y} \right)
\]
 and get $\left(\psi^{-1} \circ \hat{F}_{\alpha, \flat} \circ \psi\right)
(x, y) = \hat{F}_{\alpha, \flat}(x, y)$.
\end{enumerate}
As a consequence, we find that $\left(\psi^{-1} \circ \hat{F}_{\alpha, \flat} \circ \psi\right)(x, y) = \hat{F}(x, y)$ and $\hat{F}(x, y) \in D$ for
any $(x, y) \in D$.  Since we have chosen $\{\hat{F}^{n}(x, y) : n \ge 0\} $ and $\{\hat{F}_{\alpha, \flat}^{n}(\psi(x, y)) : n \ge 0\} $ are dense in $V_{1}$ and $V_{\alpha}$, respectively, we find that $D = V_{1}$ and $\psi(D) = V_{\alpha}$. Note that from~\eqref{eq-4-2} we see that $(x,y)\in \psi^{-1}(V_{\alpha}) = V_{1}\cap \psi^{-1} (V_{\alpha})$, then $\hat{F}(x,y)\in V_{1}\cap \psi^{-1}(V_{\alpha})$.
Choose $(x,y)\in V_{1}\cap \psi^{-1} (V_{\alpha})$ such that the orbit $\left( \hat{F}^k(x,y)\right)$ is dense in $V_{1}$. Let $(x_0,y_0)\in V_{1}$, then $(x_0,y_0) = \lim_{k\to\infty} \hat{F}^{n_k}(x,y)$ for some subsequence $(n_k)$. From the above, $\hat{F}^{n_k}(x,y)\in \psi^{-1} (V_{\alpha})$. Since $\psi^{-1} (V_{\alpha})$ is closed, taking limits we see that $(x_0,y_0)\in \psi^{-1} (V_{\alpha})$.  We conclude that
$\psi^{-1} (V_{\alpha}) = V_{1}$.  Finally we see that the choice of $(x, y)$ implies that~\eqref{eq-4-2} holds for a.e. $(x, y)$.
This concludes the assertion of this theorem.
\end{proof}
%%%%%%%%%%%%%%%%
%%%%%%%%%%%%%%%%

\section{Some applications}
As stated in the Introduction, in Section~\S5.1 we extend the result from~\cite{K-N}. That is, we show that the set of normal numbers with respect to $G_{\alpha}$ is the same with that of $G_{\alpha'}$ for any $\alpha$ and $\alpha'$ in $(0, 1]$; see Theorem~\ref{thm-4}.  To prove this result, we need the natural extensions of the $\alpha$-Farey maps.  In \S 5.2 we extend the result of~\cite{N-N-1}, by proving that for a.e.\ $\alpha$ in $(0, 1)$, $G_{\alpha}$ is not $\phi$-mixing; see \S5.2 for the definition.  To do so, we use the result of \S 5.1 together with a result from~\cite{C-T}.   What we need are statements like ``$G_{\alpha}^{n}(\alpha - 1)$ \emph{is dense}'' and ``\emph{there exist} $n_{0}$ \emph{and} $m_{0}$ \emph{such that} $G_{\alpha}^{n_{0}}(\alpha - 1) = G_{\alpha}^{m_{0}}(\alpha)$.'' The former follows from \S 5.1 and the latter from~\cite{C-T} for a.e.\ $\alpha$.

\subsection{Normal numbers}
Given any finite sequence of non-zero integers $b_1,b_2,\dots, b_n$, we define the cylinder set $\langle b_{1}, \ldots, b_{n} \rangle_{\alpha}$ as
\begin{equation}\label{eq-5-0}
\langle b_{1}, \ldots, b_{n} \rangle_{\alpha} = \{x \in \mathbb I_{\alpha} : c_{\alpha, 1}(x) = b_{1}, \ldots, c_{\alpha, n}(x) = b_{n} \},
\end{equation}
where $c_{\alpha,j}(x)=\varepsilon_{\alpha,j}(x)a_{\alpha,j}(x)$, for $j=1,2,\dots,n$. An irrational number $x\in\mathbb I_{\alpha}$ is normal with respect to $G_{\alpha}$ if for any cylinder set $\langle b_{1}, \ldots, b_{n} \rangle_{\alpha}$, 
\[
\lim_{N \to \infty} \frac{\# \{0 \le m \le N-1 : G_{\alpha}^{m}(x) \in  \langle b_{1}, \ldots, b_{n} \rangle_{\alpha} \}}{N} =
\mu_{\alpha}(\langle b_{1}, \ldots, b_{n} \rangle_{\alpha})
\]   
holds, where $\mu_{\alpha}$ is the absolutely continuous invariant probability measure for $G_{\alpha}$. An irrational number $x\in (0, 1)$ is said to be $\alpha$-normal if either $x\in [0,\alpha )$ and $x$ is normal with respect to $G_{\alpha}$, or $x\in [\alpha, 1)$ and $x-1$ is normal with respect to $G_{\alpha}$.
In the sequel, we consider $\alpha = \lim_{\epsilon\downarrow 0} (\alpha -\epsilon )$ as an element of $\mathbb I_{\alpha}$.

Now we extend this notion to the $2$-dimensional case.  Since $\hat{G}_{\alpha}^{*}$ is bijective (a.e.), we can define $\varepsilon_{\alpha,n}$ and $a_{\alpha, n}$ for $n \le 0$ which is defined by $\varepsilon_{\alpha,n}(x, y) = \varepsilon_{\alpha}({{{\hat {G^{*}_{\alpha}}}}}^{n-1}(x, y))$ and $a_{\alpha,n}(x, y) = a_{\alpha} ({{{\hat {G^{*}_{\alpha}}}}}^{n-1}(x, y))$.

We can define also $c_{\alpha, j}(x, y) = \varepsilon_{\alpha, j}(x, y) a_{\alpha, j}(x, y)$. With these definitions, we extends the notion of a $(k ,\ell)$-cylinder set for
$-\infty < k < \ell < \infty$ by:
\[
\langle b_{k}, b_{k+1}, \ldots, b_{\ell} \rangle_{\alpha, (k, \ell)} =
\{(x, y) \in \Omega_{\alpha}^{*} : c_{\alpha, k}(x) = b_{k}, \ldots,
c_{\alpha, \ell}(x) = b_{\ell} \}.
\]
Then we can define normality of an element $(x, y) \in \Omega_{\alpha}^{*}$: $(x, y)$ is said to be normal with respect to $\hat{G}^{*}_{\alpha}$ if for any sequence of integers $(b_{k}, b_{k+1}, \ldots , b_{\ell})$, $-\infty < k <\ell <\infty$
\begin{align*}
{}  &\lim_{N \to \infty} \frac{1}{N}\sharp\{1 \le n \le N :
c_{k+n-1}(x,y) = b_{k}, \ldots,c_{\ell + n-1}(x, y) = b_{\ell}\} \\
= & \,\, \hat{\mu}_{\alpha}( \langle b_{k}, b_{k+1}, \ldots , b_{\ell} \rangle_{\alpha, (k, \ell)}),
\end{align*}
where $\hat{\mu}_{\alpha}$ is the absolutely continuous invariant probability measure for $\hat{G}^{*}_{\alpha}$, satisfying $d\hat{\mu}_{\alpha} = C_{\alpha} \tfrac{dx dy}{(x - y)^{2}}$ with the normalising constant $C_{\alpha}$.
According to this definition, it is easy to see that $(x, y)$ is normal with respect to $\hat{G}^{*}_{\alpha}$ if and only if $x$ is normal with respect to $G_{\alpha}$ (independent of the choice of $y$).
For example, one may choose $y = -\infty$.  We will show the following result.\smallskip\
%%%%%%%%%%%%%%%%%%%%%%%
%%%%%%%%%%%%%%%%%%%%%%%
\begin{thm}\label{thm-4}
The set of $\alpha$-normal numbers is the same with that of $1$-normal numbers with respect to $G = G_{1}$.
\end{thm}
%%%%%%%%%%%%%%%%%%%%%%%
%%%%%%%%%%%%%%%%%%%%%%%
The proof for $\sqrt{2} - 1\leq \alpha < \tfrac{1}{2}$ is basically the same as the case $\tfrac{1}{2} \le \alpha \le 1$.
In what follows, we give the proof of this theorem mainly keeping n mind the case $0 < \alpha < \sqrt{2} - 1$.  In particular, for $0 < \alpha < \tfrac{3 - \sqrt{5}}{2}$. (By~\cite{K-S-S} and~\cite{N2022}, the size of $\Omega_{\alpha}^{*}$ with respect to the measure $\tfrac{dx dy}{(x - y)^{2}}$ is equal to that of ${\hat{G}}_{\frac{1}{2}}^{*}$ for $\tfrac{3 - \sqrt{5}}{2}\leq \alpha < \tfrac{1}{2}$ and is larger than it for $0 < \alpha < \tfrac{3 - \sqrt{5}}{2}$.)

We define an induced map $\hat{F}_{\alpha, \flat, 2}$ of $\hat{F}_{\alpha, \flat}$. To do it, first we define also an induced map $\hat{F}_{\alpha, \flat, 1}$. Here the domains of maps are $V_{\alpha, \flat, 1}$ and $V_{\alpha, \flat, 2}$,
$V_{\alpha, \flat, 1} \supset V_{\alpha, \flat, 2}$.
We put
\begin{equation}\label{Valphaflat1}
V_{\alpha, \flat, 1} = \Omega_{\alpha}^{*} \cup
\left\{ \left( - \frac{x}{1+ x}, - \frac{y}{1 + y}\right) : (x, y) \in \Omega_{\alpha}^{*},
-\tfrac{1}{2} \le x < 0 \right\} .
\end{equation}
We note that the second part of the right side is
\[
\left\{ (x, y) : x \le 1, \,\,\left(\tfrac{1}{x}, \tfrac{1}{y}\right) \in
\cup_{k=2}^{\infty}\hat{\Omega}_{\alpha, k, -} - (1, 1) \right\} .
\]
Hence we see $V_{\alpha, \flat, 1} = V_{\alpha, \flat} \cap
\{(x, y): y \le -1\}$.
Then $\hat{F}_{\alpha, \flat, 1}$ is the induced map of $\hat{F}_{\alpha, \flat}$ to
$\hat{V}_{\alpha, \flat, 1}$.
We will write it explicitly. Recall the definition of $\hat{\Omega}_{\alpha, k, \pm}$, \eqref{eq-3-1}.
\begin{enumerate}
\item
If $\alpha -1 \leq x < -\tfrac{1}{2}$, then $\hat{F}_{\alpha, \flat}(x, y) \in \Omega_{\alpha}^{*} \subset V_{\alpha, \flat, 1}$, see \eqref{Valphaflat1}, and $\hat{F}_{\alpha, \flat, 1}(x, y) = \hat{F}_{\alpha, \flat}(x, y) = \hat{G}_{\alpha}^{*}(x, y)$.

\item
If $-\tfrac{1}{2} \le x < 0$, then $\hat{F}_{\alpha}(x, y) = \left(-\tfrac{x}{1+x}, -\tfrac{y}{1+y}\right)$ which implies $ 0< -\tfrac{x}{1+x} \le 1$ and $ -\tfrac{y}{1+y} \le -1$.  Thus we have $\hat{F}_{\alpha, \flat, 1}(x, y) = \hat{F}_{\alpha, \flat}(x, y)$.

\item
If $(x, y) \in \Omega_{\alpha, k, +}^{*}$ for some $k \ge k_{0}$, then $\hat{F}_{\alpha}(x, y) = \left(\tfrac{x}{1 -x}, \tfrac{y}{1 - y}\right) \in \hat{\Omega}_{\alpha, k, +}$.  This shows $\tfrac{x}{1 - x} < 1$ but $\tfrac{y}{ 1 -y} > 1$.  Hence $\hat{F}_{\alpha}(x, y) \notin V_{\alpha, \flat, 1}$.  The same hold for $\hat{F}_{\alpha}^{\ell}(x, y)$ for $2 \le \ell \le k-1$ and then
$\hat{F}_{\alpha}^{k}(x, y) (= \hat{G}_{\alpha}^{*}(x, y)) \in \Omega_{\alpha}^{*} \subset V_{\alpha, \flat, 1}$.

\item
If $0 \le x \le 1, \,\, (\tfrac{1}{x}, \tfrac{1}{y}) \in \hat{\Omega}_{\alpha, k , -}
- (1, 1)$, then $(\tfrac{1}{x}, \tfrac{1}{y}) - (\ell, \ell) \in
\hat{\Omega}_{\alpha, k, -} - (\ell - 1, \ell -1) =: (u_{\ell}, v_{\ell})$ for $2 \le \ell
\le k-1$.  This implies $0 > -\tfrac{1}{v_{\ell}} > -1$ and so
$\left( \tfrac{1}{u_{\ell}}, \tfrac{1}{v_{\ell}}\right) \in
V_{\alpha, \flat, 1}$.  Moreover, $u_{k-1}^{-1} \in \mathbb I_{\alpha}$, where
$(\tfrac{1}{x}, \tfrac{1}{y}) - (k-1, k-1) = (u_{k}, v_{k})$.
In other words, there exists $(u', v') \in \Omega_{\alpha}^{*}$ such that
$\left(\tfrac{1}{u'}, \tfrac{1}{v'}\right) - (k, k) =
\left(\tfrac{1}{x}, \tfrac{1}{y}\right) - (k-1, k-1)$.
Hence we have 
$$
{\phantom{XXX}}\hat{F}_{\alpha, \flat, 1}(x, y) = \left(\tfrac{1}{x}- (k-1), \tfrac{1}{y} - (k-1) \right)
\left( = \hat{G}_{\alpha}^{*}(u', v') \right).
$$  
\end{enumerate}
%%%%%%%%%%%%%%%%%%%%%%%%%%
%%%%%%%%%%%%%%%%%%%%%%%%%%
%%%%%%%%%%%%%%%%%%%%%%%%%%
Consequently, we see that $\hat{F}_{\alpha, \flat, 1}(x, y)$ satisfies:
\begin{equation} \label{eq-5-1}
\left\{ \begin{array}{ll}
\hat{F}_{\alpha, \flat}(x, y) = G_{\alpha}^{*}(x,y), & \text{if } \alpha - 1 \le x < -\tfrac{1}{2} \\
\hat{F}_{\alpha, \flat}(x, y), & \text{if } -\tfrac{1}{2} \le x < 0 \\
\hat{G}_{\alpha}^{*}(x, y), & \text{if } 0 \le x \,\, \text{and} \,\, (x, y)\in  \Omega_{\alpha}^{*} \\
\left(\tfrac{1}{x} - (k - 1), \tfrac{1}{y} - (k - 1)\right) ,  & \text{if } 0 \leq x \leq 1, \text{ and}\\
 & (\tfrac{1}{x}, \tfrac{1}{y}) \in \hat{\Omega}_{\alpha, k , -}- (1, 1).
\end{array} \right.
\end{equation}

Next we consider $V_{\alpha, \flat, 2} \subset  V_{\alpha, \flat, 1}$, which is defined as follows: $V_{\alpha, \flat, 2} = V_{\alpha, \flat, -} \cup V_{\alpha, \flat, +}$, with
\[
\left\{\begin{array}{lll}
V_{\alpha, \flat, -} &= & V_{\alpha, \flat, 1}  \cap \{(x, y) : x<0, y\le -2\}\\
& = & \Omega_{\alpha}^{*}  \cap \{(x, y) : x<0, y\le -2\} ; \\
&& \\
V_{\alpha, \flat, +} & = &  V_{\alpha, \flat, 1} \cap \{(x, y) : x \ge 0\} .
\end{array} \right.
\]
Then we see that $V_{1, \flat, 2} = V_{1, \flat, +} = W$, where $W$ is defined as
\begin{equation}\label{W}
W = [0,1]\times [\infty,-1],
\end{equation}
and $V_{\alpha, \flat, -} = \emptyset$.
Recall the map $\psi$ as defined in~\eqref{PSI}, and notice that when $\psi^{-1}$ is restricted to $V_{\alpha,\flat,2}$, one finds:
\[
\psi^{-1}(x, y) = \left\{ \begin{array}{lll}
(x+1, y+1) & \text{if} & (x, y) \in V_{\alpha, \flat, -}; \\
(x, y) &\text{if} & (x, y) \in V_{\alpha, \flat, +}.
\end{array} \right.
\]
Then from Theorem~\ref{thm-3}, we have $\psi^{-1}(V_{\alpha, \flat, 2}) = W$, with $W$ from~\eqref{W}.
%%%%%%%%%%%
%%%%%%%%%%%
\begin{thm} \label{thm-5}
The induced map $\hat{F}_{\alpha, \flat, 2}$ of $\hat{F}_{\alpha, \flat, 1}$ to $V_{\alpha, \flat, 2}$ is metrically isomorphic to the natural extension of the Gauss map $G_{1}$, where the absolutely continuous invariant probability measure for $\hat{F}_{\alpha, \flat, 2}$ is given by $C_{\alpha, \flat, 2} \frac{dx dy}{(x - y)^{2}}$, with the normalising constant $C_{\alpha, \flat, 2}$, see Figures~\ref{fig-6}, \ref{fig-8}, and Figure~\ref{fig-7}.
\end{thm}
%%%%%%%%%%%%%%%%%%%%%
%%%%%%%%%%%%%%%%%%%%%
\begin{figure}
\begin{center}
\begin{tikzpicture}[scale=0.3]
\draw[thick] (-3, 0) -- (10, 0);
\node at (-2.8, 0.5){\tiny $\sqrt{2} - 2$};
\node at (-1.16, 1.5){\tiny$\frac{\sqrt{2} - 2}{2}$};
\node at (-0.3, 0.5){$0$}; \node at (1.64, 1){\tiny$\sqrt{2} - 1$};
\node at (2.8, 0.5){\tiny$\frac{\sqrt{2}}{ 2}$}; \node at (4, 0.5){$1$};

\node at (-3.5, -8)  {$-2$}; \node at (-3.5, -10.66) {$- \frac{\sqrt{5} + 3}{2}$};
\node at (-3.5, -6.45) {$-\frac{\sqrt{5}+1}{2}$};
\node at (-3.5, -4) {$-1$};
\node at (-3.5, -2.45) {$-\frac{\sqrt{5} - 1}{2}$};

\node at (-2.4, 0){$\cdot$};
\node at (-1.16, 0){$\cdot$};
\node at (0,0){$\cdot$}; \node at (1.64, 0){$\cdot$};
\node at (2.8, 0){$\cdot$};
\node at (5.6, 0){$\cdot$};  \node at (6.8, 0){$\cdot$};
\node at (9.6, 0){$\cdot$};
\draw[thick] (0, 4) -- (0, -16);
\fill[blue!40] (-2.4, -16) -- (1.66, -16)--(1.66, -10.66)--(-2.4, -10.66)--(-2.4, -16);
\fill[blue!40] (-1.16, -8) -- (1.66, -8) -- (1.66, -6.45) -- (-1.16, -6.45) -- (-1.16, -8);
\draw (9.6, -2.45)--(0, -2.45);

\draw (4, 0) -- (4, -2.45);
\fill[red!50](2.8, -4)--(0, -4)--(0, -2.45)--(2.8, -2.45)--(2.8, -4);
\fill[red!50] (4, -4) -- (0, -4)--(0, -6.46)--(4, -6.46)--(4, -4);
\draw (5.6, -4) -- (4, -4)--(4, -6.46)--(5.6, -6.46)--(5.6, -4) --cycle;

\fill[red!50] (0, -8) -- (1.66, -8) -- (1.66, -10.66) -- (0, -10.66) -- (0, -8);

\node at (4, 4){$V_{\alpha, \flat, 1}$};
\end{tikzpicture}
%%%%%%%%%%%%%%%%%%%%%%%
\qquad
\begin{tikzpicture}[scale=0.3]
\draw[thick] (-3, 0) -- (10, 0);
\node at (-2.8, 0.5){\tiny $\sqrt{2} - 2$};
%\node at (-1.16, 1){\tiny$\frac{\sqrt{2} - 2}{2}$};
\node at (-0.3, 0.5){$0$}; \node at (1.64, 1){\tiny$\sqrt{2} - 1$};
\node at (2.8, 0.5){\tiny$\frac{\sqrt{2}}{ 2}$}; \node at (4, 0.5){$1$};

\node at (-3.5, -8)  {$-2$}; \node at (-3.5, -10.66) {$- \frac{\sqrt{5} + 3}{2}$};
\node at (-3.5, -6.45) {$-\frac{\sqrt{5}+1}{2}$};
\node at (-3.5, -4) {$-1$};
\node at (-3.5, -2.45) {$-\frac{\sqrt{5} - 1}{2}$};

\node at (-2.4, 0){$\cdot$};
\node at (-1.16, 0){$\cdot$};
\node at (0,0){$\cdot$}; \node at (1.64, 0){$\cdot$};
\node at (2.8, 0){$\cdot$};
\node at (5.6, 0){$\cdot$};  \node at (6.8, 0){$\cdot$};
\node at (9.6, 0){$\cdot$};
\draw[thick] (0, 4) -- (0, -16);
\fill[blue!40] (-2.4, -16) -- (1.66, -16)--(1.66, -10.66)--(-2.4, -10.66)--(-2.4, -16);
\fill[blue!40] (0, -8) -- (1.66, -8) -- (1.66, -6.45) -- (0, -6.45) -- (0, -8);
\draw (9.6, -2.45)--(0, -2.45);
\draw (4, 0) -- (4, -2.45);
\fill[red!50](2.8, -4)--(0, -4)--(0, -2.45)--(2.8, -2.45)--(2.8, -4);
\fill[red!50] (4, -4) -- (0, -4)--(0, -6.46)--(4, -6.46)--(4, -4);
\draw (5.6, -4) -- (4, -4)--(4, -6.46)--(5.6, -6.46)--(5.6, -4) --cycle;
\draw  (-1.16, -8) -- (0, -8) -- (0, -6.45) -- (-1.16, -6.45) -- (-1.16, -8)--cycle;

\fill[red!50] (0, -8) -- (1.66, -8) -- (1.66, -10.66) -- (0, -10.66) -- (0, -8);

\node at (4, 4){$V_{\alpha, \flat, 2}$};
\end{tikzpicture}

\end{center}
\caption{$V_{\alpha, \flat, 1}$ and $V_{\alpha, \flat, 2} $ for $\alpha = \sqrt{2}-1$}
\label{fig-8}
\end{figure}
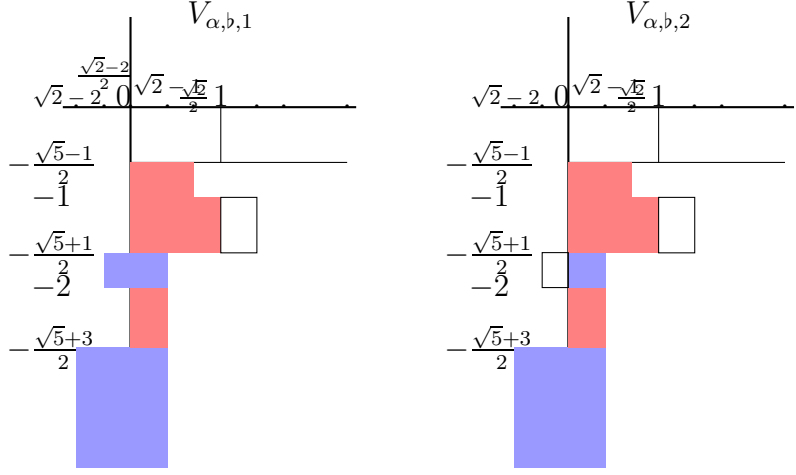
%%%%%%%%%%%%%%%%%%%%%
%%%%%%%%%%%%%%%%%%%%%

%%%%%%%%%%%%
%%%%%%%%%%%%
\begin{proof}
It is easy to see that the induced map of $\hat{F}_{1}$ on $W$ is the natural extension of the Gauss map $G$.  Indeed we see, with $W$ from~\eqref{W}, that $\hat{G}^{*} = \hat{F}_{1}|_W$:
\[
(x, y) \mapsto \left(\tfrac{1}{x} -\left\lfloor \frac{1}{x} \right\rfloor,  \tfrac{1}{y} -
\left\lfloor \frac{1}{x} \right\rfloor \right)
\]
is bijective on $W$ on (a.e.).  Since $\psi^{-1}V_{\alpha, \flat, 2} = W$, the conjugacy $\psi^{-1}\circ \hat{F}_{\alpha, \flat, 2} \circ \psi = \hat{G}^{*}_{1}$ follows from Theorem~\ref{thm-3} and basic fact on induced transformations from Ergodic theory.
\end{proof}

The next step is the definition of normal numbers associated both with $\hat{F}_{\alpha, \flat, 1}$ and $\hat{F}_{\alpha, \flat,2}$.
%%%%%%%%%%%%%%%%%%%%%
%%%%%%%%%%%%%%%%%%%%%%

We define a digit function $\delta(x, y)$ and get a sequence $(\delta_{n}(x, y):n \ge 1)$ in the following way:
\begin{equation}\label{eq-5-2}
\delta(x, y) =
\left\{ \begin{array}{ll}
\delta_{-, k} & \text{if} (x, y)\in \Omega_{\alpha, k, -}^{*}, \,\, k \ge 2 \\
\delta_{+, k} & \text{if} (x, y)\in \Omega_{\alpha, k, +}^{*}, k \ge k_{0} \\
\delta_{0, 2} & \text{if} (x, y)\in \left\{(x, y):
\left(\tfrac{1}{x}, \tfrac{1}{y}\right) \in \hat{\Omega}_{\alpha, 2, -} - (1, 1),
x \le 1 \right\}\\
\delta_{0, k} & \text{if} (x, y)\in \left\{(x, y) : \left( \tfrac{1}{x}, \tfrac{1}{y}\right) \in \hat{\Omega}_{\alpha, k, -} - (1, 1)\right\}, \,\,k > 2,
\end{array}\right.
\end{equation}
and $\delta_{n}(x, y) = \delta(\hat{F}_{\alpha, \flat, 1}^{n-1}(x, y))$, $n \ge 1$, for $(x, y) \in V_{\alpha, \flat, 1}$.  It is easy to see that the set of sequences $(\delta_{n}(x, y))$ separates points of
$V_{\alpha, \flat, 1}$.

Let $(e_{n}: 1 \le n \le \ell)$ be a block of $\delta_{j, k}$'s.
Then we define a cylinder set of length $\ell \ge 1$ by
\[
\langle e_{1}, e_{2}, \ldots, e_{\ell}\rangle_{\alpha, \flat, 1} =
\{(x, y) \in V_{\alpha, \flat, 1} : \delta_{n}(x, y) = e_{n}, 1 \le n \le \ell\} .
\]
We denote by $\mu_{\alpha, \flat, 1}$ the absolutely continuous invariant
probability measure for $\hat{F}_{\alpha, \flat, 1}$.
An element $(x, y) \in V_{\alpha, \flat, 1}$ is said to be $\alpha$-1-Farey normal if
\begin{eqnarray}\label{eq-5-3}
&&\lim_{N \to \infty} \frac{1}{N} \sharp\{n : 1 \le n \le N, \,\, \hat{F}_{\alpha, \flat,  1}^{n-1}(x, y) \in \langle e_{1}, e_{2}, \ldots, e_{\ell}
\rangle_{\alpha, \flat, 1}\} \nonumber \\
&& \qquad \,\, = \mu_{\alpha, \flat, 1}(\langle e_{1}, e_{2}, \ldots, e_{\ell}\rangle_{\alpha,\flat,1})
\end{eqnarray}
holds for every cylinder set of length $\ell \ge 1$.
%%%%%%%%%%%%%%%%%%%%%%%%%%%%%%%%%%
%%%%%%%%%%%%%%%%%%%%%%%%%%%%%%%%%%
%%%%%%%%%%%%%%%%%%%%%%%%%%%%%%%%%%
We can define the notion of the $\alpha$-2-Farey normality in a similar way, 
compare \eqref{eq-5-3} and \eqref{eq-5-3-a}.  We may use the same notation $\delta(x, y)$ restricted on $V_{\alpha, \flat, 2}$; c.f.~\eqref{eq-5-2}. However, we use $\eta(x, y)$ to describe the difference between $\hat{F}_{\alpha, \flat, 1}$ and  $\hat{F}_{\alpha, \flat, 2}$: from a digit function $\eta(x, y)$ we get a sequence $(\eta_{n}(x, y):n \ge 1)$ as follows.
\begin{equation*}
\eta(x, y) =
\left\{ \begin{array}{lll}
\delta_{-, k} & \text{if} & (x, y) \in \Omega_{\alpha, k, -}^{*}, \,\, k \ge 2 ,\,\,
y\le -2\\
\delta_{+, k} & \text{if} & (x, y) \in \Omega_{\alpha, k +}^{*}, k \ge k_{0} \\
\delta_{0, 2} & \text{if} & (x, y) \in \left\{\left(\tfrac{1}{x} , \tfrac{1}{y}
\right)
\in \Omega_{\alpha, 2, -}^{*}-(1,1), x \le 1 \right\}\\
\delta_{0, k} &\text{if}&  (x, y) \in \left\{\left(\tfrac{1}{x}, \tfrac{1}{y}
\right) \in \Omega_{\alpha, k, -}^{*} -(1,1)\right\}, \,\,k > 2
\end{array}\right.
\end{equation*}
and $\eta_{n}(x, y) = \eta(\hat{F}_{\alpha, \flat, 2}^{n-1}(x, y))$, $n \ge 1$, for
$(x, y) \in V_{\alpha, \flat, 2}$.  It is easy to see that the set of sequences
$(\eta_{n}(x, y))$ separates points of $V_{\alpha, \flat, 2}$.

Let $(e_{n} : 1 \le  n \le \ell)$ be a block $\delta_{j, k}$'s, then we define a cylinder set of length $\ell \ge 1$ by
\[
\langle e_{1}, e_{2}, \ldots, e_{\ell}\rangle_{\alpha, \flat, 2} = \{(x, y) \in V_{\alpha, \flat, 2} : \eta_{n}(x, y) = e_{n}, 1 \le n \le \ell\} .
\]
We denote by $\mu_{\alpha, \flat, 2}$ the absolutely continuous invariant
probability measure for $\hat{F}_{\alpha, \flat, 2}$.
An element $(x, y) \in V_{\alpha, \flat, 2}$ is said to be $\alpha$-2-Farey normal if
\begin{eqnarray}\label{eq-5-3-a}
&&\lim_{N \to \infty} \frac{1}{N} \sharp\{n : 1 \le n \le N, \,\, \hat{F}_{\alpha, \flat,  2}^{n-1}(x, y) \in \langle e_{1}, e_{2}, \ldots, e_{\ell} \rangle_{\alpha, \flat, 2}\} \nonumber \\
&& \qquad \,\, = \mu_{\alpha, \flat, 2}(\langle e_{1}, e_{2}, \ldots, e_{\ell}\rangle_{\alpha,\flat,2})
\end{eqnarray}
holds for every cylinder set of length $\ell \ge 1$.  Here we have to be careful with the measures $\mu_{\alpha, \flat, 1}$ and $\mu_{\alpha, \flat, 2}$, which take different values only by the normalising constants for any measurable set $A \subset V_{\alpha, \flat, 2}$.
%%%%%%%%%%%%%%%%%%%%%%%%%%%%%%%%%%%%%%%
%%%%%%%%%%%%%%%%%%%%%%%%%%%%%%%

The proof of Theorem~\ref{thm-4} is done in steps. We first prove that under the induced transformation $\hat{F}_{\alpha,\flat,1}$, $\alpha$-1-Farey normality is equivalent to $\alpha$-normality. After that we proceed to the induced system $\hat{F}_{\alpha,\flat,2}$, that is isomorphic to the Gauss map $\hat{G}_1$, and show that $\alpha$-2-Farey normality is equivalent to normality w.r.t.\ $\hat{G}_1$. On the other hand, one can prove that for points in the domain of the $\hat{F}_{\alpha,\flat,2}$ map, a point is $\alpha$-1-Farey normal if and only if it is $\alpha$-2-Farey normal. From the above equivalences, one then concludes that $\alpha$-normality is equivalent to $1$-normality.\smallskip\

Define $r_{1} := 1$ and set for $j \ge 2$, $r_{j} = r_{j}(x, y):= n $ whenever $\hat{F}_{\alpha, \flat, 1}^{n-1}(x, y) \in \Omega_{\alpha}^{*}$ and $\hat{F}_{\alpha, \flat, 1}^{m}(x, y) \notin \Omega_{\alpha}^{*}$, $r_{j-1}\leq m < n$.
%%%%%%%%%%%%%%%%
\begin{lem} \label{lem-3}
Suppose that 
$$
(x, y) \in \left\{ (x, y) : x \le 1, \,\,\left(\tfrac{1}{x}, \tfrac{1}{y}\right) \in \bigcup_{k=2}^{\infty} \hat{\Omega}_{\alpha, k, -} - (1, 1) \right\}.
$$  
Then $(x, y)$ is $\alpha$-1-Farey normal if and only if $\hat{F}_{\alpha, \flat,  1}(x, y) \in \Omega_{\alpha}^{*}$ is $\alpha$-1-Farey normal.
%%%%%%%%%%%%%%%%%%%
\end{lem}

\begin{proof} From the 4th line of the right side of~\eqref{eq-5-1}, we have $\hat{F}_{\alpha, \flat,  1}(x, y) \in \Omega_{\alpha}^{*}$.  Then the equivalence of the normality is easy to follow.
\end{proof}

From this lemma, it is enough to restrict the $\alpha$-1-Farey normality only for $(x, y) \in \Omega_{\alpha}^{*}$.
%%%%%%%%%%%%%%%%%%%%
%%%%%%%%%%%%%%%%%%%%
\begin{lem} \label{lem-4}
An element $(x_{0}, y_{0}) \in \Omega_{\alpha}^{*}$ is $\alpha$-1-Farey normal if and only if $x_{0}$ is $\alpha$-normal.
\end{lem}
%%%%%%%%%%%%%%%%%%%%
%%%%%%%%%%%%%%%%%%%%
\begin{proof}
Suppose that $r_{i} = r_{i}(x_0,y_0)$, and decompose $\N$ as $\N_{1}\cup \N_{2}$ with
$$
\N_{1} = \{ r_i\, :\, i\geq 1\} (= \{ n\in\N\, :\, \delta_{n}(x_{0}, y_{0}) = \delta_{\pm, k}\,\, \text{for some}\,\,k \}),
$$
and
$$
\N_{2} = \N\setminus\N_1 = \{ n\in\N : \delta_{n}(x_{0}, y_{0}) = \delta_{0, k}\,\, \text{for some}\,\,k \},
$$
which corresponds to
$$
\hat{F}_{\alpha, \flat, 1}^{n-1}(x_{0}, y_{0}) \in \Omega_{\alpha}^{*}\,\, \text{if}\,\, n\in\N_{1},
$$
and
$$
\hat{F}_{\alpha, \flat, 1}^{n-1}(x_{0}, y_{0}) \in \left\{(x, y) : \left(\tfrac{1}{x}, \tfrac{1}{y}\right)\in \cup_{k=2}^{\infty}
\hat{\Omega}_{\alpha, k, -} - (1, 1)\right\}\,\, \text{if}\,\, n\in\N_{2}.
$$
%%%%%%%%%%%%%%%%%%%%%%%%
%%%%%%%%%%%%%%%%%%%%%%%%
\begin{rem}\label{properties}{\rm
We easily find that the following properties hold: 
\begin{itemize}
\item[($i$)] 
$\hat{F}_{\alpha, \flat, 1}(\langle \delta_{-, 2}\rangle_{\alpha, \flat, 1} \cap \{(x, y) : -\tfrac{1}{2} \le x < 0 \} ) = \langle \delta_{0, 2} \rangle_{\alpha, \flat, 1}; $

\item[($ii$)] 
$\hat{F}_{\alpha, \flat, 1}(\langle \delta_{-, 2}\rangle_{\alpha, \flat, 1} \cap \{(x, y) : \alpha -1 \le x < -\tfrac{1}{2}  \}) = \langle \delta_{-, k} \rangle_{\alpha, \flat, 1}$,\\
for some $k\geq 2$;

\item[($iii$)] 
$\hat{F}_{\alpha, \flat, 1}(\langle \delta_{-, k}\rangle_{\alpha, \flat, 1} ) = \langle \delta_{0, k} \rangle_{\alpha, \flat, 1}, \,\,\text{for}\,\,k \ge 3$;

\item[($iv$)]
$\hat{F}_{\alpha, \flat, 1}(\langle \delta_{+, k}\rangle_{\alpha, \flat, 1} ) \subset \Omega_{\alpha}^{*}$.\medskip\
\end{itemize} 

Furthermore, if $n \in \N_{1}$ and either $\delta_{n}(x_{0}, y_{0}) = \delta_{-, k}$, $k \ge 3$, or $\delta_{n}(x_{0}, y_{0}) =  \delta_{-,2}$ and the first coordinate of $\hat{F}_{\alpha, \flat, 1}^{n-1}(x_{0},y_{0})$ is in $\left(-\tfrac{1}{2}, 0 \right)$, then $n+1\in \mathbb N_{2}$ and $n+2 \in \mathbb N_{1}$. Otherwise $n+1 \in\N_{1}$.\\

\noindent\
We note that $r_{k+1} - r_{k} = 1$ or $2$ for any $k\ge1$.  Moreover, if $\delta_{n+1}(x_{0}, y_{0}) = \delta_{0,k}$ then $\delta_{n}(x_{0}, y_{0}) = \delta_{-, k}$.\hfill $\triangle$
}\end{rem}

The properties from Remark~\ref{properties} show that we can reproduce $(\delta_{n}(x_{0}, y_{0}): j \ge 1)$ if  $(c_{j} : j \ge 1)$
is given as a sequence of integers, or equivalently, $(\delta_{r_{j}}(x_{0}, y_{0}): j \ge 1)$; c.f.~\eqref{eq-5-0}.  Indeed, for $(x, y) \in \Omega_{\alpha}^{*}$ we see the following: 
\begin{equation} \label{eq-5-3b1}
\begin{array}{lcl}
\delta_{n}(x,y) = \delta_{-, k} & \iff & \delta_{n+1}(x, y) = \delta_{0, k} \,\, \text{if} \,\, k \ge 3 \\
\delta_{n}(x,y) = \delta_{-, 2} & \Rightarrow & \left\{\begin{array}{l}
\delta_{n+1}(x, y) = \delta_{0, 2},\,\, \\
\text{and}\,\, \delta_{n+2}(x, y) = \delta_{+,k},\,\, k\geq k_0 \\
\qquad\text{or} \\
\delta_{n+ 1}(x, y) = \delta_{-,k},\,\, k\geq 2
\end{array} \right.
\end{array}
\end{equation}    
and for $\ell\geq k_0$,
\begin{equation} \label{eq-5-3b2}
\delta_{n}(x,y) = \delta_{+, \ell} \,\, \Rightarrow \,\, \left\{\begin{array}{l}
\delta_{n+1}(x, y) = \delta_{+,k} \,\, k\geq k_0\\\qquad\text{or} \\
\delta_{n+ 1}(x,y) = \delta_{-,k} \,\, k\geq 2.
\end{array} \right.
\end{equation}
So for any $(x, y) \in \Omega_{\alpha}^{*}$, from~\eqref{eq-5-3b1} and~\eqref{eq-5-3b2}, we have
\[
\begin{array}{lll}
\delta_{n}(x,y) = \delta_{0, k} & \Rightarrow & \delta_{n-1}(x, y) =
\delta(\hat{F}_{\alpha, \flat, 1}^{-1}(x, y) )= \delta_{-, k} \,\, \text{if}\,\,k\ge 3; \\
\delta_{n}(x, y)  = \delta_{0, 2} & \Rightarrow & \delta_{n-1}(x, y) = \delta_{-, 2} \,\,
\text{and} \,\, \delta_{n+1}(x, y) = \delta_{+, \ell}, \,\, \ell \geq k_{0}.
\end{array}
\]

Let $(b_{i} : 1 \leq i \leq n)$ (see~\eqref{eq-5-0}) be a sequence of non-zero integers such that $\langle b_{1}, b_{2}, \ldots, b_{n} \rangle_{\alpha} \ne \emptyset$. From the above discussion, if $(x_0,y_0)\in \langle b_{1}, b_{2}, \ldots, b_{n} \rangle_{\alpha}$, we can construct a sequence $(\hat{\delta}_{1}, \hat{\delta}_{2}, \ldots, \hat{\delta}_{m})$ satisfying
$\hat{\delta}_{r_{i}} = b_{i}$, where $r_{i}$ is the $i$th occurrence of the form
$\delta_{\pm, k}$ with $r_{1} = 1$ (since we start in $\hat{\Omega}_{\alpha}^{*}$) and $r_{n} = m$ or $m-1$, and such that $(x_0,y_0)\in \langle \hat{\delta}_1, \ldots, \hat{\delta}_m \rangle_{\alpha,\flat,1}$. The latter happens when $\hat{\delta}_{m} = \delta_{0, 2}$.  Similarly, we can construct $(b_{1}, b_{2}, \ldots, b_{n})$ from $(\hat{\delta}_{1}, \hat{\delta}_{2}, \ldots ,\hat{\delta}_{m})$, where $b_{n}= \hat{\delta}_{m}$ or $b_{n} = \hat{\delta}_{m-1}$.

To show the statement of the lemma, first we assume that $(x_{0}, y_{0})$ is $\alpha$-$1$-Farey normal.
It is easy to see that:
\begin{eqnarray} \label{eq-5-4}
&&\lim_{N \to \infty}\frac{1}{N} \sharp\{ n : 1 \le n \le N, n \in \mathbb N_{2}\}  = \lim_{K \to \infty} \mu_{\alpha, \flat, 1}(\cup_{k=2}^{K} \langle \delta_{0,k} \rangle_{\alpha, \flat,1}) \notag \\
&& \qquad \,\, = \mu_{\alpha, \flat, 1}(\cup_{k=2}^{\infty} \langle \delta_{0,k} \rangle_{\alpha, \flat,1}) .
\end{eqnarray}
The equality~\eqref{eq-5-4} also shows that
\begin{equation} \label{eq-5-4a}
\lim_{N \to \infty}\frac{1}{N} \sharp\{ n : 1 \le n \le N, n \in \mathbb N_{1}\}
=
\mu_{\alpha, \flat, 1}(\hat{\Omega}_{\alpha}^{*} ) ,
\end{equation}
By the same argument, we can show that for any sequence $(\hat{\delta}_{1},\ldots ,\hat{\delta}_{m})$,
$$
\mu_{\alpha, \flat,1}(\langle \hat{\delta}_{1}, \hat{\delta}_{2}, \ldots , \hat{\delta}_{m} \rangle_{\alpha, \flat, 1}) = \mu_{\alpha, \flat, 1}(\langle b_{1},b_{2}, \ldots , b_{n} \rangle_{\alpha})
$$
if $\hat{\delta}_{m}$ is of the form $\delta_{\pm, k}$, otherwise
$$
\mu_{\alpha, \flat,1}(\langle \hat{\delta}_{1}, \hat{\delta}_{2}, \ldots , \hat{\delta}_{m} \rangle_{\alpha, \flat, 1}) =  \sum_{\ell_{0}}^{\infty} \mu_{\alpha, \flat, 1}(\langle b_{1}, b_{2}, \ldots , b_{n}, \ell \rangle_{\alpha}).
$$

We have
\begin{align} \label{eq-5-4b}
{} & \frac{1}{N} \sharp\{ j : 1 \le k\le N, \delta_{j} = b_{1}, \delta_{j+1}  = b_{2},
\ldots , \delta_{ n} = b_{n}\} \notag \\
= & \frac{\hat{N}}{N} \frac{1}{\hat{N}}
\sharp\{ j : 1 \le j\le \hat{N}, c_{j} = b_{1}, c_{j+1} = b_{2}, \ldots,  c_{j+ n -1} = b_{n}\}
\end{align}
where $\hat{N}= \max \{j : r_{j} \le N \}$ and $c_{j} = \varepsilon_{j}(x_{0}
\cdot a_{j}(x_{0})$.

With this notation, the left side converges to $\mu_{\alpha, \flat, 1}(\langle \hat{\delta}_{1}, \ldots, \hat{\delta}_{m}\rangle_{\alpha, \flat, 1} )$ and the first term of the right side goes to $\hat{\mu}_{\alpha, \flat, 1}
(\Omega_{\alpha}^{*})$ (see~\eqref{eq-5-4a}).
Thus the second term of the right side goes to
\[
\frac{\mu_{\alpha, \flat, 1}(\langle \hat{\delta}_{1}, \hat{\delta}_{2},  \ldots,
\hat{\delta}_{m}\rangle_{\alpha, \flat, 1} )}
{\mu_{\alpha, \flat, 1}(\Omega_{\alpha}^{*})} .
\]
as $N \to \infty$ and its numerator is
$$
\mu_{\alpha, \flat, 1}(\{(x, y) : c_{r_{1}}(c, y) = b_{1}, c_{r_{2}}(x, y) = b_{2}, \ldots, c_{r_{n}}(x, y) = b_{n}\}.
$$
The definition of $\mu_{\alpha}$ and $\mu_{\alpha, \flat, 1}$ implies that
$\tfrac{1}{\mu_{\alpha, \flat,1}(\hat{\Omega}_{\alpha}^{*})}$ changes
$\mu_{\alpha, \flat, 1}$ to $\mu_{\alpha}$.  Consequently,
we get the limit of the second term as
$\mu_{\alpha}(\langle b_{1}, b_{2}, \ldots,  b_{n} \rangle_{\alpha})$. This shows that$(x_{0}, y_{0})$ is normal with respect to $\hat{G}^{*}_{\alpha}$, which implies the $\alpha$-normality of $x_{0}$.

Next we suppose that $(x_{0}, y_{0})$ is not $\alpha$-$1$-Farey normal.
We want to show that $(x_{0}, y_{0})$ is also not $\alpha$-normal.
We check the equality~\eqref{eq-5-4a} again.  If $\tfrac{\hat{N}}{N}$ does not
conveges to $\mu_{\alpha, \flat,1}(\Omega_{\alpha}^{*})$, then it is easy to
see that $x_{0}$ is not $\alpha$-normal.  On the other hand if
$\lim_{N \to \infty} \tfrac{\hat{N}}{N} = \mu_{\alpha, \flat,1}
(\Omega_{\alpha}^{*})$, then, by the same argument, we see that
the second term of the right side of \eqref{eq-5-4b} does not converge to
$\mu_{\alpha}(\langle b_{1}, b_{2}, \ldots,  b_{n} \rangle_{\alpha})$, which shows
that $x_{0}$ is not $\alpha$-normal.
Indeed, from Remark~\ref{properties} we can construct a sequence $(\hat{\delta}_{1}, \hat{\delta}_{2}, \ldots , \hat{\delta}_{m})$ such that
$\hat{\delta}_{r_{j}} = \delta_{{\rm sgn}(b_{j}), \, |b_{j}|}$ and $m = r_{J }= \hat{N}$,  i.e., $\hat{\delta}_{r_{j}} = \delta_{{\rm sgn}, |b_{j}| }$, $1 \le j \le J$ and for other $\ell$ ($\ell \ne r_{j}$, $1 \le j \le J$) are of the form $\delta_{0, k}$, and it is determined uniquely by $\hat{\delta}_{\ell - 1}$ since $\ell -1$ is $r_{j}$ for some $1 \le j \le J$.  Indeed, $\hat{\delta}_{\ell} = \delta_{0, k}$ implies $\hat{\delta}_{\ell -1} = \delta_{-, k}$.  Note that $\hat{\delta}_{\ell -1} = \delta_{-, 2}$ does not mean $\hat{\delta}_{\ell} = \delta_{0, 2}$ since $\ell -1 = r_{j}$, $\ell = r_{j+1}$ can happen.  But  $\hat{\delta}_{\ell -1} = \delta_{-, k}$, $k \ge 3$, implies $r_{j}+1 \ne r_{j+1}$. Then we can show the estimate in the above.
\end{proof}

%%%%%%%%%%%%%%%%%%%%
%%%%%%%%%%%%%%%%%%%%
The same idea shows that the $\alpha$-1-Farey normality is equivalent to the $\alpha$-2-Farey normality.  Here we note that $\hat{F}_{\alpha, \flat,2}$ is an induced map of $\hat{F}_{\alpha, \flat, 1}$ and for $(x,y)\in V_{\alpha, \flat, 2}$, $\eta_n(x,y)\eta_{n+1}(x,y)\neq \delta_{- , 2}\delta_{0,2}$. So in the $\eta$-code of a point $(x,y)$ the digit $\delta_{0,2}$
serves as a marker for the missing preceding digit $\delta_{-,2}$ in the corresponding $\delta$-code of $(x,y)$.
%%%%%%%%%%%%%%%%%%%%
%%%%%%%%%%%%%%%%%%%%
\begin{lem} \label{lem-5}
An element $(x, y) \in \Omega_{\alpha}^{*}$ is $\alpha$-2-Farey normal if and only if it is $\alpha$-1-Farey normal.
\end{lem}

{\it Sketch of the proof.}
From the sequence $\delta_{n}(x, y)$, we can construct $\eta_{m}(x, y) \in V_{\alpha, \flat, 2}$ by deleting the digit $\delta_{-,2}$ that is followed by a digit $\delta_{0,2}$. More precisely,
\begin{eqnarray*}
&&\delta_{n-1}(x, y), \delta_{n}(x, y) = \delta_{-,2},\,\, \delta_{n+1}((x, y) = \delta_{0,-2}  \\
&\Rightarrow& \eta_{m} = \delta_{n-1}(x, y), \eta_{m+1}(x, y) = \delta_{n+1},
\end{eqnarray*}
where $m$ is the cardinality of $n$ such that $\delta_{n}\delta_{n+1} = \delta_{-, 2}\delta_{0,2}$.  On the other hand, given the sequence $(\eta_{m}(x, y) : - \infty < m< \infty)$, we can construct the sequence $(\delta_{n}(x, y) : -\infty < n < \infty)$ by inserting $\delta_{-,2}$ before every occurrence of every $\delta_{0,2}$.  Following the proof of Lemma~\ref{lem-4}, we get the result. \hfill $\Box$\\

%%%%%%%%%%%%%%%%%%%%
%%%%%%%%%%%%%%%%%%%%
\begin{lem}  \label{lem-6}
An element $(x_{0}, y_{0}) \in \Omega_{\alpha}^{*}$ is $\alpha$-2-normal if and only if $\psi^{-1}(x_{0}, y_{0}) \in W$ is normal with respect to $\hat{G}^{*}$.
\end{lem}
%%%%%%%%%%%%%%%%%%%%
%%%%%%%%%%%%%%%%%%%%
\begin{proof}
Following the above proofs, cylinder sets associated with $\hat{F}_{\alpha, \flat,2}$ are approximated by each other (using $\psi$ and $\psi^{-1}$); see Theorem \ref{thm-5}. Suppose that $(x_{0}, y_{0})$ is $\alpha$-2-Farey normal. Every cylinder set associated with $\hat{G}^{*}(= \hat{G}_{1}^{*})$ is a rectangle. To be more precise, a cylinder set $\langle a_{1}, a_{2}, \ldots, a_{n} \rangle_{1}$ is of the form
$\left[\frac{p_{n}}{q_{n}}, \frac{p_{n} + p_{n-1}}{q_{n} + q_{n-1}} \right) \times
[-\infty, -1]$, or $\left(\frac{p_{n} + p_{n-1}}{q_{n} + q_{n-1}},
\frac{p_{n}}{q_{n}} \right] \times [-\infty, -1]$. We divide it into three parts such that
$\eta_{1}(x, y) = \delta_{\sharp, k}$, $\sharp = +, 0$, and $-$.
Then $\psi^{-1}$-image of each part is a countable union of
$\hat{F}_{\alpha, \flat,2}$ cylinder sets, just as discussed in the above.
Hence we can prove that $(x_{0}, y_{0})$ (or $(x_{0}+1, y_{0}+1)$ is normal with
respect to $\hat{G}^{*}$ in the same way.

Now suppose that $(x_{0}, y_{0})$ is \textbf{not} $\alpha$-2-normal and $\psi^{-1}(x_{0}, y_{0})$ is normal with respect to $\hat{G}^{*}$.
Then there exist $\epsilon > 0$ and a cylinder set with respect to $\hat{F}_{\alpha, \flat, 2}$ such that either
\begin{eqnarray} \label{eq-5-5}
&&\lim_{N \to \infty} \frac{1}{N} \# \{ n : 0 \le m \le N-1, \,\,
\hat{F}_{\alpha, \flat,  2}^{m}(x_0, y_0) \in \langle e_{1}, e_{2}, \ldots, e_{\ell}
\rangle_{\alpha, \flat, 2}\} \nonumber \\
&& \qquad \,\, > \mu_{\alpha, \flat, 2}(\langle e_{1}, e_{2}, \ldots, e_{\ell}\rangle_{\alpha,\flat,2}) + \epsilon ,
\end{eqnarray}
or
\begin{eqnarray}\label{eq-5-5-2}
&&\lim_{N \to \infty} \frac{1}{N} \# \{ n : 0 \le m \le N-1, \,\,
\hat{F}_{\alpha, \flat,  2}^{m}(x_0, y_0) \in \langle e_{1}, e_{2}, \ldots, e_{\ell}
\rangle_{\alpha, \flat, 2}\} \nonumber \\
&& \qquad \,\, < \mu_{\alpha, \flat, 2}(\langle e_{1}, e_{2}, \ldots, e_{\ell}\rangle_{\alpha,\flat,2}) - \epsilon ,
\end{eqnarray}
and
\begin{eqnarray}\label{eq-5-6}
&&\lim_{N\to\infty} \frac{1}{N} \# \{0\leq m\leq N-1:\hat{G^{*}}^{m}(\psi^{-1}(x_{0},
y_{0})) \in \langle b_{1}, \ldots, b_{n} \rangle_{1,(1,n)} \} \nonumber \\
&& \qquad \,\, = \hat{\mu}(\langle b_{1}, \ldots, b_{n} \rangle_{1, (1, n)}),
\end{eqnarray}
for any sequence of positive integers $(b_{1}, b_{2}, \ldots, b_{n})$, where $\hat{\mu}$ is the measure defined by $\tfrac{1}{\log 2} \tfrac{dx dy}{(x - y)^{2}}$.
Note that $\langle \cdots \rangle_{1, (1,n)}$ means a cylinder set with respect to $\hat{G}^{*} = \hat{G}_{\alpha}^{*}$ with $\alpha = 1$.

We start by assuming~\eqref{eq-5-5} and~\eqref{eq-5-6} hold, and show it will lead to a contradiction. Since the set of cylinder sets associated with $\hat{G}^{*}$ generates the Borel $\sigma$-algebra, there exist a finite number of pairwise disjoint cylinder sets
\[
\langle b_{j,1}, b_{j, 2}, \ldots, b_{j, k_{j}}\rangle_{1,(1, k_{j})},
\qquad 1 \le j \le M <\infty,
\qquad 1 \le k_{j} < \infty
\]
such that,
\[
\psi^{-1}(\langle e_{1}, e_{2}, \ldots, e_{\ell} \rangle_{\alpha, \flat, 2}) \subset
\bigcup_{j = 1}^{M}
\langle b_{j,1}, b_{j, 2}, \ldots, b_{j, k_{j}}\rangle_{1,(1, k_{j})}
\]
and 
\[%begin{equation}   \label{eq-5-7}
\hat{\mu} \left( \bigcup_{j = 1}^{M}
\langle b_{j,1}, b_{j, 2}, \ldots, b_{j, k_{j}}\rangle_{1}\right)
<
\hat{\mu}(\psi^{-1}(\langle e_{1}, e_{2}, \ldots, e_{\ell} \rangle_{\alpha, \flat, 2})) + \tfrac{1}{2}\epsilon .
\]  %end{equation}

Since $\psi$ is an isomorphism (see Theorem~\ref{thm-5}), we have
$$
\hat{F}_{\alpha,\flat,2}(x_0,y_0) = \psi \hat{G}^{*}\psi^{-1}(x_0,y_0),
$$
and
$$
\mu_{\alpha, \flat, 2}(\langle e_{1}, e_{2}, \ldots, e_{\ell}\rangle_{\alpha,\flat,2}) = \hat{\mu}(\psi^{-1}(\langle e_{1}, e_{2}, \ldots, e_{\ell} \rangle_{\alpha, \flat, 2})).
$$
Thus,
\begin{eqnarray*}
&& \lim_{N \to \infty} \frac{1}{N} \# \{ 0 \leq m \leq N-1:\,\, \hat{F}_{\alpha, \flat,  2}^{m}(x_0,y_0)\in \langle e_{1}, e_{2}, \ldots, e_{\ell} \rangle_{\alpha, \flat, 2}\} \\
&\leq & \lim_{N \to \infty} \frac{1}{N} \# \{0 \leq m \leq N-1:\,\, \hat{G^{*}}^{m}(\psi^{-1}(x_{0},y_{0})) \in \bigcup_{j = 1}^{M} \langle b_{j,1}, \ldots, b_{j, k_{j}}\rangle_{1} \}\\
&=& \hat{\mu}\left( \bigcup_{j = 1}^{M} \langle b_{j,1}, b_{j, 2}, \ldots, b_{j, k_{j}}\rangle_{1}\right)\\
&<& \hat{\mu}(\psi^{-1}(\langle e_{1}, e_{2}, \ldots, e_{\ell} \rangle_{\alpha, \flat, 2})) + \tfrac{1}{2}\epsilon \\
&=& \mu_{\alpha, \flat, 2}(\langle e_{1}, e_{2}, \ldots, e_{\ell}\rangle_{\alpha,\flat,2}) + \tfrac{1}{2}\epsilon .
\end{eqnarray*}
Combining this with~\eqref{eq-5-5} yields $\epsilon < \tfrac{1}{2}\epsilon$, which is a contradiction.\smallskip\

On the other hand, if~\eqref{eq-5-5-2} and~\eqref{eq-5-6} hold, a proof similar to the one above but now approximating the cylinder $\psi^{-1}\left(\langle e_{1}, e_{2}, \ldots, e_{\ell}\rangle_{\alpha,\flat,2}\right)$ from ``inside'' by a union of cylinders leads to the same contradiction.
\end{proof}

{\it Proof of Theorem \ref{thm-4}.}  This is a direct consequence of Lemmas \ref{lem-3}, \ref{lem-4}, \ref{lem-5} and \ref{lem-6}.
\qed

%%%%%%%%%%%%%
%%%%%%%%%%%%%
\subsection{Non-$\phi$-mixing property}
We start with some definitions of mixing properties. Let $(\Omega, \mathfrak B, P)$ be a probability space.  For sub $\sigma$-algebras
$\mathcal A$ and $\mathcal B \subset \mathfrak B$, we put
\[
\phi (\mathcal A, \mathcal B) = \sup
\left\{
\left|\frac{P(A \cap B)}{P(A)} - P(B) \right| : A \in \mathcal A, \,\, B \in \mathcal B, \,\, P(A)>0\right\} .
\]

Suppose that $(X_{n} : n \ge 1)$ is a stationary sequence of random variables. We denote by $\mathcal F_{m}^{n}$ the sub-$\sigma$ algebra of $\mathfrak B$ generated by $X_{m}, X_{m+1}, X_{m+2}, \ldots ,  X_{n}$. We define $\phi(n) = \sup_{m \ge 1} \phi(\mathcal F_{1}^{m}, \mathcal F_{m+n}^{\infty})$ and $\phi^{*}(n) = \sup_{m \ge 1} \phi(\mathcal F_{m+n}^{\infty}, \mathcal F_{1}^{m})$.

The process $(X_{n}:n \ge 1)$ is said to be $\phi$-mixing if $\lim_{n \to \infty} \phi(n) = 0$ and reverse $\phi$-mixing if $\lim_{n \to \infty} \phi^{*}(n) = 0$, respectively.

$G_{\alpha}$ is said to be $\phi$-mixing (or reverse $\phi$-mixing) if $(a_{\alpha, n}, \varepsilon_{\alpha, n})$ is $\phi$-mixing (or reverse $\phi$-mixing), respectively. In \cite{N-N-1}, it is shown that $G_{\alpha}$ is not $\phi$-mixing for a.e.\ $\alpha$, $\tfrac{1}{2} \le \alpha \le 1$.  On the other hand, $G_{\alpha}$ is reverse $\phi$-mixing for every $\alpha$, $0 < \alpha \le 1$, which follows from~\cite{A-N}.

In~\cite{N-N-1}, it is shown that $G_{\alpha}$ is weak Bernoulli for any $\tfrac{1}{2}\leq \alpha \le 1$ but is not $\phi$-mixing.  It is not hard to show that $G_{\alpha}$ is weak Bernoulli for any $0 < \alpha < \tfrac{1}{2}$ following the proof given in \cite{N-N-1}.  On the other hand, it follows that $G_{\alpha}$ is reverse $\phi$-mixing for any $0 < \alpha \le 1$; see \cite{A-N}.  In the proof of the next Theorem we outline how one can extend the proofs of~\cite{N-N-1} to the case $0<\alpha < \tfrac{1}{2}$.

\begin{thm} \label{thm-6}
For almost every $\alpha$, $0 < \alpha < 1$, $G_{\alpha}$ is not $\phi$-mixing.
\end{thm}

\emph{Sketch of the proof.} The proof of the non-$\phi$-property in \cite{N-N-1} is based on the following two properties:
\begin{enumerate}
\item[($i$)]
For almost every $\alpha$, $\tfrac{1}{2} \leq \alpha \leq 1$, $(G_{\alpha}^{n}(\alpha) : n \ge 0)$ is dense in $\mathbb I_{\alpha}$.

\item[($ii$)]
For every $\alpha$, $\tfrac{1}{2} < \alpha < \tfrac{\sqrt{5} - 1}{2}$, $G_{\alpha}^2(\alpha) = G_{\alpha}^{2}(\alpha -1)$ and for every $\alpha$, $\tfrac{\sqrt{5}-1}{2} \le \alpha < 1$, $G_{\alpha}^{2}(\alpha) = G_{\alpha}(\alpha - 1)$, respectively.
\end{enumerate}
The first statement follows from the fact that the set of normal numbers w.r.t.\ $\alpha$ is independent of $\alpha$ (\cite{K-N}).  Because of Theorem~\ref{thm-4} above, we can extend ($i$) to almost every $0 < \alpha \le 1$.

The second statement is generalized in~\cite{C-T}: for almost every $\alpha$, there exists $n, m$ such that $G_{\alpha}^{n}(\alpha) = G_{\alpha}^{m}(\alpha - 1)$. From this, we can show that thin cylinders exist for almost every $\alpha$ and for any $\epsilon > 0$. To be more precise, for any $\epsilon > 0$, a cylinder set $\mathcal C = \langle c_{\alpha, 1}, c_{\alpha, 2}, \ldots , c_{\alpha, \ell} \rangle_{\alpha}$ is said to be an $\epsilon$-thin-cylinder if:
\begin{itemize}
\item[a)] $G_{\alpha}^{\ell}(\mathcal C)$ is an interval;
\item[b)] $G_{\alpha}^{\ell}:\, \mathcal C \to G_{\alpha}^{\ell}(\mathcal C)$ is bijective;
\end{itemize}
and
\begin{itemize}
\item[c)] $|G_{\alpha}^{\ell}(\mathcal C)| < \epsilon$.
\end{itemize}
Once we have a sequence of $\epsilon_{n}$-thin cylinders with $\epsilon_{n}
\searrow 0$, the proof is completely the same as one given in \cite{N-N-1} if we
choose $\alpha$ so that the matching property holds and
$\alpha - 1$ is $\alpha$-normal. For in this case, there exist $n_{0}$, $m_{0}$ such that $G_{\alpha}(\alpha - 1)^{n_{0}}
= G_{\alpha}^{m_{0}}(\alpha)$ (matching property). Moreover, there exists $n_{\epsilon} > \max (n_{0}, m_{0})$ such that $\min (|\alpha - G_{\alpha}^{n_{\epsilon}}(\alpha -1)|, |(\alpha - 1) - G_{\alpha}^{n_{\epsilon}}(\alpha -1)| < \epsilon)$, which follows from the normality of $\alpha -1$.

We suppose that $|(\alpha - 1) - G_{\alpha}(\alpha -1)| < \epsilon$. Because of the matching property, see~\cite{C-T}, either $\langle c_{1}(\alpha - 1), c_{2}(\alpha - 1), \ldots,
c_{n_{\epsilon}}(\alpha -1) \rangle_{\alpha}$  or $\langle c_{1}(\alpha), c_{2}(\alpha), \ldots,
c_{n_{\epsilon}}(\alpha) \rangle_{\alpha}$ is an $\epsilon$-thin-cylinder set.
 This is because of the following: If $\alpha$ is normal, then it means $\alpha$ is not rational nor quadratic.  The iteration $G_{\alpha}^{n}$ associated with $\alpha$
and $G_{\alpha}^{m}$ associated with $\alpha - 1$ are linear fractional transformations.  Hence $\alpha \mapsto G_{\alpha}^{n}(\alpha)$ and
$(\alpha \mapsto G_{\alpha}^{m}\circ `` - 1")(\alpha)$ define the same linear fractional transformation, otherwise $\alpha$ is
a fixed point of a linear fractional transformation which means $\alpha$ is either rational or quadratic.  We denote by $L_{r}$, $L_{\ell}$ and $S$ the linear fractional transformations which induce $G_{\alpha}^{n}(\alpha)$, $G_{\alpha}^{m}(\alpha -1)$ and $x \mapsto x - 1$, respectively.  Then $L_{r}(\langle c_{1}(\alpha),\ldots, c_{n_{\epsilon}}(\alpha) \rangle_{\alpha}) = (L_{\ell}\circ S) (\langle c_{1}(\alpha), \ldots, c_{n_{\epsilon}}(\alpha) \rangle_{\alpha}) $.  This shows that 
$G_{\alpha}^{n}(\langle c_{1}(\alpha),  \ldots, c_{n_{\epsilon}}(\alpha) \rangle_{\alpha}) $ and
$G_{\alpha}^{m}(\langle c_{1}(\alpha - 1), \ldots,
c_{n_{\epsilon}}(\alpha -1) \rangle_{\alpha})$ have one common end point
$G_{\alpha}^{n}(\alpha)$ and no common inner point.
In the case of $|\alpha - G_{\alpha}(\alpha -1)| < \epsilon$,  the same holds exactly.
In this way, we can choose a sequence of $\epsilon_{n}$-thin cylinders and Theorem~\ref{thm-6} follows in exactly the same way as in~\cite{N-N-1}.\hfill $\Box$

\section*{Acknowledgement}
This research was partially supported by JSPS grants 20K03642 and 24K03611.

\end{document}